\documentclass[11pt]{amsart}
\usepackage[utf8]{inputenc}
\usepackage[T1]{fontenc}
\usepackage{amssymb,amsmath}
\usepackage{bm}
\usepackage{graphicx}
\usepackage{color}
\usepackage{tikz,tikz-3dplot}
\usepackage{pgfplots}
\usepackage{enumitem}
\usepackage{algorithm}
\usepackage{setspace}
\usepackage{algpseudocode}
\usepackage{mathtools}
\usepackage[hidelinks]{hyperref}
\usepackage{xfrac}

\usepackage{accents}
\usepackage{multicol} 
\usepackage[autoplay]{animate}
\usepackage{soul}
\usepackage{subcaption}
\newtheorem{theorem}{Theorem}

\theoremstyle{definition}

\theoremstyle{lemma}
\newtheorem{lemma}[theorem]{Lemma}

\theoremstyle{remark}
\newtheorem{remark}[theorem]{Remark}

\newtheorem{assumption}[theorem]{Assumption}
\numberwithin{theorem}{section}
\numberwithin{equation}{section}
\numberwithin{table}{section}
\numberwithin{figure}{section}
\newcommand{\V}{\ensuremath{\mathcal{V}}}
\newcommand{\Q}{\ensuremath{\mathcal{Q}}}
\newcommand{\M}{\ensuremath{\mathcal{M}}}

\def\R{\mathbb{R}}

\definecolor{myBlue1}{RGB}{101,149,239}  
\definecolor{myBlue2}{RGB}{113,104,238} 
\definecolor{myBlue3}{RGB}{30,144,255} 
\definecolor{myGreen1}{RGB}{154,204,50} 
\definecolor{myGreen2}{RGB}{69,169,0} 
\definecolor{myGreen3}{RGB}{154,205,50} 
\definecolor{myGreen4}{RGB}{105,139,34} 
\definecolor{myRed1}{RGB}{210,105,30} 
\definecolor{myRed2}{RGB}{165,42,42} 
\definecolor{myRed3}{RGB}{139,26,26} 
\definecolor{lightgray}{RGB}{175,175,175} 
\definecolor{myLGray}{RGB}{225,225,225} 
\DeclareMathOperator{\id}{id}
\DeclareMathOperator{\tr}{tr}

\newcommand{\masterB}{B}

\newcommand{\Vh}{\mathbb{V}_h}
\newcommand{\Qh}{\mathbb{Q}_h}
\newcommand{\Mh}{\mathbb{M}_h}
\newcommand{\Xh}{\mathbb{X}_h}
\newcommand{\Ein}{\calE_\Omega^\text{in}}
\newcommand{\Ebd}{\calE_\Omega^\text{bd}}
\newcommand{\eg}{D}

\newcommand{\calKu}{\calK_u}
\newcommand{\calKp}{\calK_p}
\newcommand{\calBu}{\calB_u}
\newcommand{\calBp}{\calB_p}

\newcommand{\calB}{\ensuremath{\mathcal{B}} }

\newcommand{\calE}{\ensuremath{\mathcal{E}} }

\newcommand{\calK}{\ensuremath{\mathcal{K}} }
\newcommand{\calL}{\ensuremath{\mathcal{L}} }
\newcommand{\calM}{\ensuremath{\mathcal{M}} }
\newcommand{\calN}{\ensuremath{\mathcal{N}} }

\newcommand{\calP}{\ensuremath{\mathcal{P}} }
\newcommand{\calQ}{\ensuremath{\mathcal{Q}} }

\newcommand{\calT}{\ensuremath{\mathcal{T}} }

\newcommand{\calV}{\ensuremath{\mathcal{V}} }

\newcommand{\calX}{\ensuremath{\mathcal{X}} }

\newcommand{\fu}{f}
\newcommand{\fp}{g}
\newcommand{\au}{\alpha}
\newcommand{\ap}{\kappa}

\def\ds{\,\text{d}s}

\def\dx{\,\text{d}x}

\def\dy{\,\text{d}y}

\newcommand{\partialan}{\ensuremath{\partial_{\au,\nu}} }

\allowdisplaybreaks
\begin{document}
\title[AFEM for parabolic problems with dynamic boundary conditions]{A posteriori error estimation for parabolic problems with dynamic boundary conditions}
\author[]{R.~Altmann$^\dagger$, C.~Zimmer$^{\ddagger}$}
\address{${}^{\dagger}$ Institute of Analysis and Numerics, Otto von Guericke University Magdeburg, Universit\"atsplatz 2, 39106 Magdeburg, Germany}
\address{${}^{\ddagger}$ Department of Mathematics, University of Augsburg, Universit\"atsstr.~12a, 86159 Augsburg, Germany}
\email{robert.altmann@ovgu.de,christoph.zimmer@uni-a.de}
\thanks{Both authors acknowledge the support of the Deut\-sche For\-schungs\-ge\-mein\-schaft (DFG, German Research Foundation) through the project 446856041. Moreover, major parts of this work were carried out while the first author was affiliated with the Institute of Mathematics and the Centre for Advanced Analytics and Predictive Sciences (CAAPS) at the University of Augsburg.}
%
%
\date{\today}
\keywords{}
\begin{abstract}
This paper is concerned with adaptive mesh refinement strategies for the spatial discretization of parabolic problems with dynamic boundary conditions. This includes the characterization of inf--sup stable discretization schemes for a stationary model problem as a preliminary step. Based on an alternative formulation of the system as a partial differential--algebraic equation, we introduce a posteriori error estimators which allow local refinements as well as a special treatment of the boundary. We prove reliability and efficiency of the estimators and illustrate their performance in several numerical experiments. 
\end{abstract}
%
%
\maketitle
\setcounter{tocdepth}{3}
%
{\tiny {\bf Key words.} dynamic boundary conditions, adaptivity, PDAE, parabolic PDE}\\
\indent
{\tiny {\bf AMS subject classifications.}  {\bf 65M60}, {\bf 65L80}, {\bf 65J10}} 
%
%
%
\section{Introduction}

Within this paper, we consider a linear parabolic model problem with so-called {\em dynamic boundary conditions} in a bounded and polyhedral (Lipschitz) domain $\Omega\subseteq \R^d$, $d\in\{2,3\}$, on a time horizon $[0,T]$, $0<T<\infty$, namely  
\begin{subequations}
\label{eq:parabolicDynBC}
\begin{align}
	\dot u - \Delta_\alpha u
	&= \hat\fu \qquad\text{in } \Omega, \label{eq:parabolicDynBC:eqn} \\
	\dot u -   \Delta_{\Gamma,\kappa} u + \partialan u 
	&= \hat\fp \qquad\text{on } \Gamma \coloneqq \partial\Omega  \label{eq:parabolicDynBC:bc}
\end{align}
\end{subequations}
with initial condition~$u(0) = u^0$ and sufficiently smooth right-hand sides~$\hat\fu$, $\hat\fp$. Therein, we write 
\[
  \Delta_\alpha u \coloneqq \nabla\cdot(\au \nabla u), \qquad 
  \Delta_{\Gamma,\kappa} u \coloneqq \nabla_\Gamma\cdot(\ap \nabla_\Gamma u)
\]
for the weighted Laplacians with uniformly positive diffusion coefficients~$\au\in L^\infty(\Omega)$ and $\ap\in L^\infty(\Gamma)$. In the special case~$\ap\equiv 1$, the differential operator in~\eqref{eq:parabolicDynBC:bc} equals the well-known Laplace--Beltrami operator; see~\cite[Ch.~16.1]{GilT01}. Moreover, $\partialan u \coloneqq \nu\cdot(\au \nabla u)$ denotes the normal derivative corresponding to the differential operator in~\eqref{eq:parabolicDynBC:eqn} with unit normal vector~$\nu$. 

Condition~\eqref{eq:parabolicDynBC:bc} itself is a differential equation and enables the reflection of effective properties on the boundary of the domain~\cite{Esc93,Gol06}. To be more precise, this means that the momentum on the boundary is taken into account during the modeling process rather than being neglected as for standard boundary conditions. Because of this flexibility, this and similar models have attained increasing popularity nowadays, cf.~\cite{Gol06,Vit18,LiuW19}. Although well-understood from a theoretical point (see, e.g., \cite{FavGGR02,VazV08}), the numerical approximation of such systems is only tackled by a small number of papers; see~\cite{KovL17,VraS13a} as well as the early work~\cite{Fai79}. 

Within this paper, we benefit from a special formulation of~\eqref{eq:parabolicDynBC} as {\em partial differential--algebraic equation} (PDAE); see~\cite{LamMT13,Alt15} for an introduction. This means that we interpret the equations as a coupled system with an additional variable acting only on the boundary; see~\cite[Ch.~5.3]{Las02} as well as~\cite{Alt19}. The resulting PDAE approach has already proven advantageous from a numerical point of view in the context of heterogeneous boundary conditions~\cite{AltV21}, phase separation models such as the Cahn--Hilliard equation~\cite{AltZ22ppt_c}, and the design of bulk--surface splitting schemes~\cite{AltKZ22,AltZ22ppt_b,CsoFK23}. The main instrument of the PDAE approach is the possibility to apply different discretization schemes in the bulk~$\Omega$ and on the boundary~$\Gamma$, e.g., based on different meshes $\calT_\Omega$ and~$\calT_\Gamma$. This is also exploited in the present paper and is of special interest in applications where the solution oscillates rapidly on the boundary. 
Such oscillations also call for adaptive mesh refinements in order to improve the computational efficiency. Adaptive strategies, on the other hand, are based on a posteriori error estimates, which were first introduced for elliptic problems; see~\cite{AinT97,Lip04,Ver13}. 

The paper is structured as follows. After the introduction of the PDAE model in Section~\ref{sect:formulation:weak}, we shortly consider the temporal discretization and introduce a stationary model problem with a saddle point structure. Afterwards, we discuss inf--sup stable finite element schemes for this model problem in Section~\ref{sect:formulation:space}. Based on classical finite element literature, we construct (local) a posteriori error estimators in Section~\ref{sect:aposteriori}. Of special emphasis is the possibility to distinguish needed refinements of the bulk mesh on the boundary (denoted by $\calT_\Omega|_\Gamma$) and the boundary mesh~$\calT_\Gamma$. Moreover, we prove that these estimators are reliable as well as efficient, meaning that the estimators are indeed in the range of the actual error. We complete the paper with a number of numerical experiments in Section~\ref{sect:numerics}. This includes stationary problems on different spatial domains but also a dynamic problem such as~\eqref{eq:parabolicDynBC}. 

\subsection*{Notation}  
Throughout the paper we write~$a \lesssim b$ to indicate that there exists a generic constant~$C$, independent of any spatial and temporal discretization parameters, such that~$a \leq C b$. 
%
%
\section{Weak Formulation and Discretization}
\label{sect:formulation}

In this preliminary section, we introduce the weak operator formulation of the model problem~\eqref{eq:parabolicDynBC} in PDAE form. Moreover, we discuss the discretization in time and space. 
%
\subsection{Weak formulation}
\label{sect:formulation:weak}

With the introduction of an auxiliary variable $p\coloneqq \tr u$ (with~$\tr$ denoting the usual trace operator) as proposed in~\cite{Alt19} and the spaces 
\[
  \V \coloneqq H^1(\Omega), \qquad
  \Q \coloneqq H^1(\Gamma), \qquad 
  \M \coloneqq H^{-\sfrac 12}(\Gamma),
\]  
the weak form of~\eqref{eq:parabolicDynBC} can be expressed as the following PDAE: Find $u\colon [0,T] \to \V$, $p\colon [0,T] \to \Q$, and a Lagrange multiplier~$\lambda\colon [0,T] \to \M$ such that   
\begin{subequations}
\label{eq:PDAE}
\begin{align}
	\begin{bmatrix} \dot u \\ \dot p  \end{bmatrix}
	+ \begin{bmatrix} \calKu &  \\  & \calKp \end{bmatrix}
	\begin{bmatrix} u \\ p \end{bmatrix}
	+ \begin{bmatrix} -\calBu^* \\ \calBp^* \end{bmatrix} \lambda 
	&= \begin{bmatrix} \hat\fu \\ \hat\fp \end{bmatrix} \qquad \text{in } (\V\times \Q)^*, \label{eq:PDAE:a} \\
	\calBu u - \calBp p 
	&= \phantom{[[} 0\hspace{2.9em} \text{in } \M^*. \label{eq:PDAE:b}
\end{align}
\end{subequations}
This formulation includes the differential operators~$\calKu\colon \V \to \V^*$ and~$\calKp\colon \Q \to \Q^*$ given by
\[
	\langle \calKu u, v\rangle 
	\coloneqq \int_\Omega \au \nabla u \cdot \nabla v \dx, \qquad 
	\langle \calKp p, q\rangle
	\coloneqq \int_\Gamma \ap\, \nabla_\Gamma p \cdot \nabla_\Gamma q \dx  
\]
as well as the trace operator~$\calBu\colon \V\to\M^*$ and the canonical embedding~$\calBp\colon \Q\to\M^*$.  

We would like to emphasize that~\eqref{eq:PDAE} depicts an extended formulation where the connection of~$u$ and its trace~$p$ is explicitly stated in the form of a constraint. In contrast to the classical formulation~\eqref{eq:parabolicDynBC}, the PDAE~\eqref{eq:PDAE} includes a Lagrange multiplier as additional variable which satisfies $\lambda=\partial_{\alpha,\nu} u$ if the solution is sufficiently smooth. Moreover, the system comes together with initial data~$u(0) = u^0$ and~$p(0) = p^0$, which is called {\em consistent} if they coincide on the boundary, i.e., if $\calBu u^0 = \calBp p^0$. The existence of a unique (distributional) solution of system~\eqref{eq:PDAE} for consistent initial data follows from the inf--sup stability of the constraint operator together with the G\aa rding inequality of the differential operators; see~\cite{Alt19} as well as~\cite{EmmM13}.

At this point, we would like to emphasize that a direct spatial discretization of~\eqref{eq:PDAE} leads to a differential--algebraic equation (of differentiation index~$2$). In general, this implies numerical difficulties within the temporal discretization, cf.~\cite{HaiLR89,HaiW96}. Since the right-hand side of the constraint~\eqref{eq:PDAE:b} is homogeneous, however, these instabilities do not occur for consistent initial values and an index reduction is not needed, cf.~\cite[p.~33]{HaiLR89}. 

Within this paper, we follow the {\em Rothe method} for the discretization of system~\eqref{eq:PDAE}, i.e., we first discretize in time and then in space. The time discretization is shortly discussed in the following subsection, whereas we will focus on the spatial discretization using adaptive finite elements in each time step in Section~\ref{sect:formulation:space}. 
%
%
\subsection{Temporal discretization}
\label{sect:formulation:time}

We first consider the application of the implicit Euler scheme to~\eqref{eq:PDAE} with constant step size~$\tau$. Given approximations~$u^n\in\V$ and~$p^n\in\Q$ of $u$ and $p$ at time point $t^n \coloneqq n\tau$, respectively, we seek $u^{n+1}\in\V$, $p^{n+1}\in\Q$, and~$\lambda^{n+1}\in\M$ as the solution of
\begin{subequations}
\label{eq:PDAE:Euler}
\begin{align}
	\begin{bmatrix} u^{n+1} \\ p^{n+1} \end{bmatrix}
	+ \begin{bmatrix} \tau\calKu &  \\  & \tau\calKp \end{bmatrix}
	\begin{bmatrix} u^{n+1} \\ p^{n+1} \end{bmatrix}
	+ \begin{bmatrix} -\tau\calBu^* \\ \tau\calBp^* \end{bmatrix} \lambda^{n+1} 
	&= \begin{bmatrix} \tilde \fu \\ \tilde \fp \end{bmatrix} \qquad \text{in } (\V\times \Q)^*, \label{eq:PDAE:Euler:a} \\
	\calBu u^{n+1} - \calBp p^{n+1} 
	&= \phantom{[} 0\hspace{3.1em} \text{in } \M^*. \label{eq:PDAE:Euler:b}
\end{align}
\end{subequations}
Therein, $\tilde \fu$ (and similarly $\tilde \fp$) is some right-hand side which depends on the previous approximation~$u^n$ (respectively $p^n$). 
\begin{remark}
For semilinear applications, i.e., if we allow $\fu$ and $\fp$ to depend on $u$ and $p$, respectively, we obtain the same linear structure as long as the nonlinear terms are treated {\em explicitly}. 
\end{remark}
In summary, an implicit Euler discretization leads to the following (stationary) problem, which needs to be solved in every time step: Given $\fu\in \V^*$ and $\fp\in \Q^*$, find~$u\in \V$, $p\in \Q$, and $\lambda\in\M$ such that  
\begin{align*}
	\begin{bmatrix} u \\ p \end{bmatrix}
	+ \begin{bmatrix} \tau\calKu & \\ & \tau\calKp \end{bmatrix}
	\begin{bmatrix} u \\ p \end{bmatrix}
	+ \begin{bmatrix} -\tau\calBu^* \\ \tau\calBp^* \end{bmatrix} \lambda
	&= \begin{bmatrix} \fu \\ \fp \end{bmatrix} \qquad \text{in } (\V\times \Q)^*, \\
	\calBu u - \calBp p 
	&= \phantom{[} 0\hspace{3.1em} \text{in } \M^*.
\end{align*}
This motivates the model problem, which we will consider in Section~\ref{sect:aposteriori}, namely  
\begin{subequations}
	\label{eqn:saddle_point_AFEM}
	\begin{alignat}{4}
	(\sigma - \Delta_\alpha)\, u &= \fu, \\
	(\sigma - \Delta_{\Gamma,\kappa})\, p + \lambda &= \fp, \\
	\tr u - p  &= 0. 
	\end{alignat}
\end{subequations}
Note that this model is written in its strong form because of which the Lagrange multiplier (being equal to the weighted normal trace) disappears in the first equation. Moreover, the newly introduced parameter~$\sigma>0$ corresponds to $\tau^{-1}$ in the case of an implicit Euler discretization.
\begin{remark}[Discretization by Runge--Kutta methods]\label{rem:RK}
Assume a Butcher tableau
\begin{equation*}
	\begin{array}{c|c}
		\mathbf{c} & \mathbf{A}\\ \hline & \mathbf{b}^T
	\end{array} 
\end{equation*}
with an invertible coefficient matrix $\mathbf A \in \R^{s,s}$, which defines an algebraically stable Runge--Kutta scheme. According to~\cite[Sect.~8.3 f.]{Zim21}, this leads to a stable approximation. Furthermore, we assume that $-\mathbf A$ is globally stable, which is equivalent to the existence of a symmetric matrix $\mathbf D \in \R^{s,s}$ such that~$\mathbf D$ and $\mathbf{DA}^{-1}+\mathbf{A}^{-T}\mathbf{D}$ are positive definite~\cite[Th. 13.1.1]{LanT85}. Then, the calculation of the internal stages $\mathbf{u}^{n+1}\in \calV^s$, $\mathbf{p}^{n+1}\in \calQ^s$, $\bm{\lambda}^{n+1}\in \calM^s$ can be treated similarly to the implicit Euler method considered above. In comparison to~\eqref{eq:PDAE:Euler}, $u^{n+1}$ (and analogously~$p^{n+1}$ and~$\lambda^{n+1}$) needs to be replaced by $\mathbf{u}^{n+1}$. The operators are replaced by their Kronecker product with~$\mathbf D$, e.g., $\calK$ turns into $\mathbf D \otimes \calK$. Moreover, the operator $\mathbf D \mathbf{A}^{-1} \otimes \id$ pops up in front of~$\mathbf{u}^{n+1}$ and~$\mathbf{p}^{n+1}$. Finally, $\tilde \fu$ (and analogously~$\tilde \fp$) contains the approximation of the previous time point $u^n = (1-\mathbf b^T \mathbf A^{-1} \mathbf e) u^{n-1} + \mathbf b^T \mathbf A^{-1}\mathbf u^n$, $n=1,2,\ldots$, with $\mathbf e=[1,\ldots,1]^T \in \R^s$; cf.~\cite[Ch.~5 \& 8]{Zim21}.
	
For the translation of the upcoming analysis given in Section~\ref{sect:aposteriori} to the Runge--Kutta case, one uses that the assumptions on the matrix~$\mathbf D$ imply
\begin{gather*}
	\|\mathbf u^{n+1}\|^2_{[L^2(\Omega)]^s} 
	\lesssim (\mathbf D \mathbf A^{-1} \mathbf u^{n+1},\mathbf u^{n+1})_{[L^2(\Omega)]^s} 
	\lesssim \|\mathbf u^{n+1}\|^2_{[L^2(\Omega)]^s},\\
	\vert\mathbf u^{n+1}\vert^2_{[H^1(\Omega)]^s} 
	\lesssim \langle(\mathbf D \otimes \calK_u) \mathbf u^{n+1},\mathbf u^{n+1}\rangle 
	\lesssim \vert\mathbf u^{n+1}\vert^2_{[H^1(\Omega)]^s}.
\end{gather*}
Examples of diagonal matrices~$\mathbf D$ for the Gauss--Legendre, Radau IA, and Radau IIA methods can be found at \cite[p.~220 f.]{HaiW96}. 
\end{remark}
%
%
\subsection{Stable spatial discretization}
\label{sect:formulation:space}

We first collect a number of standard finite element spaces, which will be used in the following. Given a regular triangulation~$\calT_{\Omega}$ of the computational domain~$\Omega$, the space of (discontinuous) piecewise polynomials of degree~$k\geq 0$ is denoted by 
\begin{equation*}
	\calP_k^\text{d}(\calT_{\Omega}) 
	\coloneqq \big\{ v \in L^2(\Omega)\,\big|\, v|_{T} \text{ is a polynomial of degree $\leq k$ for all } T \in \calT_{\Omega} \big\}.
\end{equation*}
The subspace of functions which are additionally elements of $H^1(\Omega)$, is denoted by $\calP_k(\calT_{\Omega})$. It is well-known that this space can be characterized by
\begin{equation*}
	\calP_k(\calT_{\Omega}) 
	= \calP_k^\text{d}(\calT_{\Omega}) \cap C(\Omega);
\end{equation*}
see, e.g., \cite[Th.~II.5.2]{Bra07}. Given a regular triangulation $\calT_\Gamma$ of the boundary $\Gamma$ we denote analogously the space of (possibly discontinuous) piecewise polynomial functions and its subspace of continuous functions by $\calP_k^\text{d}(\calT_{\Gamma})$ and $\calP_k(\calT_{\Gamma})$, respectively. 

For the construction of stable finite element schemes, we further introduce the space of {\em bubble functions}. In two space dimensions, the edge-bubble function $\psi_E$ equals the scaled product of the two nodal basis functions corresponding to the nodes of the edge $E$. For $d = 3$, $\psi_E$ denotes the face-bubble ($E$ being a face in $\calT_{\Omega}$) and equals the scaled product of the three corresponding nodal basis functions, cf.~\cite{Ver13}. This leads to the space
\begin{multline*}
  \calE_\ell(\calT_\Omega) 
  \coloneqq \big\{ v \cdot \psi_E\,\big|\, v|_T \text{ is a polynomial of degree} \leq \ell \text{ for all } T \in \calT_\Omega,\\*
  \psi_E \text{ is an edge/face-bubble for } E \subseteq \Gamma \big\} 
  \subseteq P_{\ell+d}(\calT_{\Omega}).
\end{multline*}

Now let~$\Vh$, $\Qh$, and $\Mh$ denote finite-dimensional subspaces of $\V$, $\Q$, and $\M$. Due to the saddle point structure of the model problem, these spaces need to satisfy a discrete inf--sup condition, which reads 
\begin{equation}\label{eqn:disc_inf_sup}
  \adjustlimits\inf_{\lambda_h\in \Mh\setminus\{0\}} \sup_{(u_h,p_h)\in \Vh\times \Qh\setminus\{0\}} 
  \frac{\langle \calB_u u_h-\calB_p p_h,\lambda_h\rangle_{\Gamma}}{\|\lambda_h\|_{\calM}\, \big(\|u_h\|^2_{\calV}+\|p_h\|^2_{\calQ}\big)^{\sfrac 12}} 
  \geq \beta 
  > 0
\end{equation}
for some positive constant $\beta$ being independent of the mesh sizes. Here, $\langle\,\cdot\,,\cdot\,\rangle_\Gamma$ denotes the dual product in~$\calM$, which only acts on $\Gamma$. In a first result, we show that this stability condition is independent of the choice of the space~$\Qh$. 
\begin{theorem}[Equivalence of inf--sup stable discretizations]
\label{th:inf_sup_equivalence}
Consider a conforming Galerkin scheme~$\Vh \subseteq \calV = H^1(\Omega)$, $\Qh \subseteq \calQ = H^1(\Gamma)$, and $\Mh \subseteq \calM = H^{-\sfrac 12}(\Gamma)$, which includes the approximation property of the family~$\Vh$. Then, the discrete inf--sup condition~\eqref{eqn:disc_inf_sup} is satisfied if and only if there exists a positive constant $\hat{\beta}$ (independent of the mesh size) such that
\begin{equation}\label{eqn:disc_inf_sup_u_la}
	\adjustlimits \inf_{\lambda_h\in \Mh\setminus\{0\}} \sup_{u_h\in \Vh\setminus\{0\}} \frac{\langle \calB_u u_h,\lambda_h\rangle_\Gamma}{\|\lambda_h\|_{\calM}\|u_h\|_{\calV}} \geq \hat{\beta}
	> 0. 
\end{equation}
\end{theorem}
\begin{proof}
($\Rightarrow$) We prove the if-part by contrapositive and  assume that~\eqref{eqn:disc_inf_sup_u_la} is not satisfied. Then, there exists a sequence $\{\lambda^\ast_h\} \subseteq \calM$ with $\lambda^\ast_h\in\Mh$ and~$\|\lambda_h^\ast\|_{\calM}=1$ such that
\begin{equation*}
	a_h \coloneqq \sup_{u_h\in \Vh\setminus\{0\}} \frac{\langle \calB_u u_h,\lambda_h^\ast\rangle_\Gamma}{\|u_h\|_{\calV}} 
	\to 0
	\qquad\quad\text{as }
	h \to 0^+.
\end{equation*}
Since the sequence is uniformly bounded, there exists a subsequence $h^\prime$ of $h$ with $\lambda_{h^\prime}^\ast \rightharpoonup \lambda^\ast$ in $\calM$ as $h^\prime \to 0^+$. 
Let $g \in \calM^\ast\setminus \{0\}$ be arbitrary. The surjectivity of $\calB_u\in \calL(\calV,\calM^\ast)$, \cite[Ch.~III, Th.~4.2]{Bra07}, implies the existence of a~$u \in \calV$ with $\calB_u u = g$. Due to the assumed approximation property of $\Vh$, there exists a sequence $\{u_h\}_h \subseteq \calV$ with $u_h \in \Vh$, $u_h\neq 0$, and $\lim_{h \to 0^+} u_h = u$ in~$\calV$. By the strong convergence of $\{u_h\}_h$ and the weak convergence of~$\{\lambda_{h^\prime}\}_{h^\prime}$, we have 
\begin{align*}
	0\leq \vert \langle g, \lambda^\ast\rangle_\Gamma \vert 
	= \lim_{h^\prime \to 0^+} \vert \langle \calB_u u_{h^\prime},\lambda_{h^\prime}^\ast\rangle_\Gamma \vert
	= \lim_{h^\prime \to 0^+}\|u\|_{\calV} \frac{\vert \langle \calB_u u_{h^\prime},\lambda_{h^\prime}^\ast\rangle_\Gamma\vert }{\|u_{h^\prime}\|_{\calV}} 
	\leq \lim_{h^\prime \to 0^+}\|u\|_{\calV}\, a_{h^{\prime}} = 0.
\end{align*}
Hence, $\lambda^\ast=0$ holds and the entire sequence~$\{\lambda^\ast_{h}\}$ vanishes weakly in $\calM=H^{-\sfrac 12}(\Gamma)$, cf.~\cite[Ch. I, Lem. 5.4]{GajGZ74}. Moreover, the sequence vanishes strongly in~$\calQ^\ast = H^{-1}(\Gamma)$, since the weak limit is unique and the embedding $H^{-\sfrac 12}(\Gamma) \hookrightarrow H^{-1}(\Gamma)$ is compact~\cite[Prop.~3.2 \&~3.4]{Tay11}. Finally, this leads to
\begin{align*}
	\lim_{h\to 0^+} &\adjustlimits\inf_{\lambda_h\in \Mh\setminus\{0\}} \sup_{(u_h,p_h)\in \Vh\times \Qh\setminus\{0\}}  \frac{\langle \calB_u u_h-\calB_p p_h,\lambda_h\rangle_\Gamma}{\|\lambda_h\|_{\calM}\,\big(\|u_h\|^2_{\calV}+\|p_h\|^2_{\calQ}\big)^{\sfrac 12}}\\
	&\leq \lim_{h\to 0^+} \sup_{(u_h,p_h)\in \Vh\times \Qh\setminus\{0\}} \frac{\langle \calB_u u_h-\calB_p p_h,\lambda_h^\ast \rangle_\Gamma}{\big(\|u_h\|^2_{\calV}+\|p_h\|^2_{\calQ}\big)^{\sfrac 12}}\\
	&\leq \lim_{h\to 0^+} \sup_{u_h\in \Vh\setminus\{0\}} \frac{\langle \calB_u u_h,\lambda_h^\ast \rangle_\Gamma}{\|u_h\|_{\calV}} +\sup_{p_h\in \Qh\setminus\{0\}} \frac{\langle\calB_p p_h,\lambda_h^\ast \rangle_\Gamma}{\|p_h\|_{\calQ}} \\
	&\leq \lim_{h\to 0^+} a_h + \|\lambda_h^\ast \|_{\calQ^\ast} 
	= 0,
\end{align*}
where we have used that $\calB_p$ is the embedding operator of $\calQ=H^{1}(\Gamma)$ into $\calM^\ast = H^{\sfrac 12}(\Gamma)$. Thus, the inf--sup expression cannot be bounded uniformly from below by a $\beta>0$.
	
($\Leftarrow$) The only-if-part follows immediately by
\begin{equation*}
	\sup_{(u_h,p_h)\in \Vh\times \Qh\setminus\{0\}}   \frac{\langle \calB_u u_h-\calB_p p_h,\lambda_h\rangle_\Gamma}{\big(\|u_h\|^2_{\calV}+\|p_h\|^2_{\calQ}\big)^{\sfrac 12}}
	\geq \sup_{u_h\in \Vh\setminus\{0\}} \frac{\langle \calB_u u_h,\lambda_h\rangle_\Gamma}{\|u_h\|_{\calV}}\geq \hat{\beta}\, \|\lambda_h\|_{\calM}.
	\qedhere
\end{equation*}
\end{proof}
\begin{remark}
In the limit case $\kappa\equiv0$, there is no differential operator acting on the boundary and the boundary conditions are called {\em locally reacting}. The corresponding solution spaces then read~$\calV = H^1(\Omega)$, $\calQ = L^2(\Gamma)$, and $\calM = L^{2}(\Gamma)$. In this setting, one can show analogously to Theorem~\ref{th:inf_sup_equivalence} that conforming finite element spaces are inf--sup stable if and only if the subproblem with $\Vh = \{0\}$ is inf--sup stable. 
\end{remark}
A direct consequence of Theorem~\ref{th:inf_sup_equivalence} is that it is sufficient to consider the inf--sup condition with $p_h=0$. Hence, one example class of inf--sup stable discretizations reads 
\begin{equation*}
	\Vh \coloneqq \calP_k(\calT_{\Omega})\cup \calE_\ell(\calT_{\Omega}) 
	\subseteq \calV, \qquad 
	\Qh \subseteq \calQ, \qquad 
	\Mh\coloneqq \calP_\ell^{\text{d}}(\calT_{\Omega}|_{\Gamma}) 
	\subseteq \calM
\end{equation*}
with $k\geq 1$, $\ell\geq 0$; see~\cite[Th.~2.3.7]{Lip04}. We would like to emphasize that the mesh used for the definition of~$\Mh$ may be different to~$\calT_\Omega|_\Gamma$ as long as it is synchronized with the mesh used for the bubble functions.  
Moreover, the scheme
\begin{equation*}
	\Vh \coloneqq \calP_1(\calT_{\Omega}) \subseteq \calV, \qquad 
	\Qh \subseteq \calQ, \qquad 
	\Mh\coloneqq \calP_1(\calT_{\Omega}|_{\Gamma}) \subseteq \calM
\end{equation*}
is inf--sup stable; see~\cite[Prop.~3.4]{AltV21}. The corresponding generalization to higher polynomial degrees reads as follows.
\begin{lemma}
The conforming finite element spaces
\begin{equation*}
	\Vh \coloneqq \calP_k(\calT_{\Omega}) \subseteq \calV, \qquad 
	\Qh \subseteq \calQ, \qquad 
	\Mh\coloneqq \calP_k(\calT_{\Omega}|_{\Gamma}) \subseteq \calM
\end{equation*}
satisfy the discrete inf--sup condition~\eqref{eqn:disc_inf_sup} for arbitrary $k\geq 1$ and $\Qh$.
\end{lemma}
\begin{proof}
The case $k=1$ is shown in \cite[Prop.~3.4]{AltV21}. The statement for general $k$ can be proven equivalently. Here, we use that there exists an extension operator from $\calP_k(\calT_{\Omega}|_{\Gamma}) \subseteq H^{\sfrac 1 2}(\Omega)$ to $\calP_k(\calT_{\Omega}) \subseteq H^1(\Omega)$ which is continuous and preserves boundary values; cf.~\cite[Lem.~5.1]{HipM12} and~\cite[Th.~2.1]{ScoZ90}.
In particular, its operator-norm depends only on $\Omega$ and the regularity of the triangulation $\calT_{\Omega}$.
\end{proof}
Recall that all presented stable schemes still have the freedom to set the discrete space~$\Qh$. In the following, $\Qh$ will be a standard finite element space but defined on another triangulation~$\calT_\Gamma$. Having in mind applications with strong fluctuations or heterogeneities on the boundary, $\calT_\Gamma$ usually equals a refinement of $\calT_\Omega|_\Gamma$.  
%
%
\section{A Posteriori Error Estimation}
\label{sect:aposteriori}

After the discussion on the temporal discretization and inf--sup stable finite element schemes, we now turn to the construction of efficient and reliable a posteriori error estimators. Following the Rothe method, the aim is an adaptive spatial discretization in each time step. Note that the formulation as coupled system in~\eqref{eq:PDAE} amplifies the use of adaptive schemes as we can optimize $\calT_\Omega$ (for the spaces $\Vh$, $\Mh$) as well as $\calT_\Gamma$ (for $\Qh$). Throughout this paper, we assume~$\calT_\Omega$ and $\calT_\Gamma$ to be shape regular; see~\cite[Ch.~II.5]{Bra07}. Note that this automatically implies the shape regularity of $\calT_{\Omega}|_\Gamma$. Within this section, we consider the two-dimensional setting but comment on necessary adjustments in the three-dimensional case. 

By $\calE_\Omega$ we denote the set of edges corresponding to the triangulation~$\calT_\Omega$, which can be decomposed into~$\Ein$ (interior edges) and $\Ebd$ (boundary edges). Moreover, $h_T$, $h_I$, and~$h_E$ denote the mesh sizes for elements and edges of a triangulation, respectively. For the upcoming analysis, we consider the following assumption. 
\begin{assumption}[Spatial discretization]
\label{ass:disc}
~
\begin{enumerate}
	\item The discretization scheme is inf--sup stable and it holds that $\calP_1(\calT_\Omega) \subseteq \Vh$, $\calP_1(\calT_\Gamma) \subseteq \Qh$, and $\calP_1(\calT_\Omega|_\Gamma) \subseteq \Mh$ or $\calP_0^\text{d}(\calT_\Omega|_\Gamma) \subseteq \Mh$.
	\item $\calT_{\Gamma}$ is a refinement of $\calT_\Omega|_{\Gamma}$. Moreover, there exists a positive constant $\rho>0$ such that $h_E \leq \rho h_I$ holds uniformly for every $I \in \calT_\Gamma$ with $I \subseteq E \in \Ebd$. 
	\item The diffusion coefficients satisfy $\alpha \in \calP_1^\text{d}(\calT_{\Omega})$ and $\kappa \in \calP_1^\text{d}(\calT_{\Gamma})$.   
\end{enumerate}
\end{assumption}
%

As explained in Section~\ref{sect:formulation:time}, we deal with the model problem~\eqref{eqn:saddle_point_AFEM}. Working again with the weak form, we consider the system 
\begin{alignat*}{4}
	\sigma (u, v)_{L^2(\Omega)} + \langle \calKu u, v\rangle - \langle \lambda, v\rangle_\Gamma 
	&= (\fu, v)_{L^2(\Omega)}, \\
	\sigma (p, q)_{L^2(\Gamma)} + \langle \calKp p, q\rangle + \langle \lambda, q\rangle_\Gamma 
	&= (\fp, q)_{L^2(\Gamma)}, \\
	\langle u - p, \mu\rangle_\Gamma 
	&= 0
\end{alignat*}
for test functions $v\in\V$, $q\in\Q$, $\mu\in\M$ and the parameter $\sigma$ which is related to the time step size. 
\begin{remark}
The upcoming analysis can be performed similarly for semilinear applications; see~\cite[Sect.~5.2.3]{Ver13}. We, however, focus on the linear case, since our main interest is the treatment of the two different meshes on the boundary.
\end{remark}

To shorten notation, we define the product space~$\calX \coloneqq \calV \times \calQ \times \calM$ equipped with the norm 
\[
	\|[u,p,\lambda]\|^2 
	\coloneqq \|u\|^2_{\calV} + \|p\|^2_{\calQ} + \|\lambda\|^2_{\calM}.
\]
Moreover, we introduce the bilinear form $\masterB\colon \calX \times \calX \to \R$, which is directly related to the model problem, by 
\begin{align*}
	\masterB\big(&[u,p,\lambda], [v,q,\mu]\big)\\*
	&= \int_\Omega \sigma uv + \au \nabla u\cdot \nabla v \dx 
	+ \int_\Gamma \sigma pq + \ap\, \nabla_\Gamma p\cdot \nabla_\Gamma q \ds 
	+ \int_\Gamma (u-p)\mu - (v-q)\lambda \ds.
\end{align*}
Note that we use the symbolic integral notation for the boundary terms as they equal the $L^2$-inner product if the arguments are sufficiently smooth. Obviously, the bilinear form~$\masterB$ is continuous. Furthermore, following the lines of \cite[Prop.~4.68]{Ver13}, one can show that~$\masterB$ is inf--sup stable, i.e., there exists a positive constant~$\beta_{\masterB} > 0$ such that   
\begin{equation}
	\label{eqn:inf_sup_masterB}
	\sup_{[v,q,\mu]\in \calX,\ \|[v,q,\mu]\|=1 } \masterB\big([u,p,\lambda],[v,q,\mu]\big) 
	\geq \beta_{\masterB} \|[u,p,\lambda]\|. 
\end{equation}
With $\masterB$ at hand, the introduced weak form of the model problem can be written as 
\[
	\masterB\big([u,p,\lambda], [v,q,\mu]\big) 
	= \int_\Omega \fu v \dx + \int_\Gamma \fp q \ds
\]
for all $[v,q,\mu]\in \calX$. Correspondingly, given discrete spaces~$\Vh\subset\V$, $\Qh\subset\Q$, and~$\Mh\subset\M$ and the product space~$\Xh \coloneqq \Vh \times \Qh \times \Mh$, we obtain a discrete solution triple $[u_h,p_h,\lambda_h] \in \Xh$ as the unique solution of 
\[
	\masterB\big([u_h,p_h,\lambda_h], [v_h,q_h,\mu_h]\big) 
	= \int_\Omega \fu v_h \dx + \int_\Gamma \fp q_h \ds
\]
for all $[v_h,q_h,\mu_h]\in \Xh$. Standard arguments such as in~\cite[Sect.~II.2.2]{BreF91} imply the a priori estimate 
\[
	\big\|[u-u_h,p-p_h,\lambda-\lambda_h]\big\|
	\lesssim \inf_{[v_h,q_h,\mu_h] \in \Xh}  \|u-v_h\|_{H^1(\Omega)} + \|p-q_h\|_{H^{1}(\Gamma)} + \|\lambda-\mu_h\|_{H^{-\sfrac 12}(\Gamma)}.
\]
%
%
At this point, we would like to recall that we are especially interested in cases where strong fluctuations occur on the boundary. This means, in particular, that we expect $\Qh$ to be a refinement of $\Vh$ restricted to the boundary. 
Moreover, this motivates the use of adaptive methods, which are based on a posteriori error estimates. These are topic of the upcoming subsection.
%
%
\subsection{Construction of local estimators}
\label{sect:aposteriori:def}

Following the methodology of~\cite{Lip04,Ver13}, we aim to construct a posteriori error estimators which can distinguish necessary refinements for the approximations of $u|_\Gamma$ and $p$, respectively. 
Within this section, the triple $[u,p,\lambda]\in\calX$ denotes the weak solution to the model problem~\eqref{eqn:saddle_point_AFEM} and $[u_h,p_h,\lambda_h]\in\Xh$ its discrete counterpart. As a consequence of~\eqref{eqn:inf_sup_masterB}, we have the error bound 
\begin{multline*}
	\|u-u_h\|_{H^1(\Omega)} + \|p-p_h\|_{H^1(\Gamma)} + \|\lambda-\lambda_h\|_{H^{-1/2}(\Gamma)} \\*
	\leq \frac{\sqrt 3}{\beta_{\masterB}} \sup_{[v,q,\mu]\in \calX,\ \|[v,q,\mu]\|=1 }\masterB\big([u-u_h,p-p_h,\lambda-\lambda_h],[v,q,\mu]\big).
\end{multline*}
In order to bound the right-hand side, we integrate by parts and obtain
\begin{align*}
	&\masterB\big([u-u_h,p-p_h,\lambda-\lambda_h],[v,q,\mu]\big)\\
	&= \sum_{T \in \calT_\Omega} \int_T \big(f - \sigma u_h + \Delta_\alpha u_h \big)\, v \dx 
	- \sum_{E \in \Ein} \int_E\big[\alpha \nabla u_h \cdot n_E \big]_E\, v \ds  \\* 
	&\quad+ \sum_{E \in \Ebd} \int_E \big( \lambda_h - \alpha \nabla u_h \cdot n_E \big)\, v \ds + \sum_{I \in \calT_\Gamma} \int_I \big(g - \sigma p_h + \Delta_{\Gamma,\kappa} p_h - \lambda_h\big)\, q \ds \\
	&\quad 
	- \sum_{x\in \calN_\Gamma} \big[ \kappa\nabla_\Gamma p_h \big]_x\, q(x)
	- \int_\Gamma \big( u_h - p_h \big)\, \mu \ds
\end{align*}
for every $[v,q,\mu]\in \calX$. Recall that $\Ein$ and $\Ebd$ denote the set of interior and boundary edges of the triangulation~$\calT_\Omega$, respectively. Moreover, $\calN_\Gamma$ equals the set of nodes of~$\calT_\Gamma$ and~$[\, \cdot\, ]_E$ the jump along an edge~$E$.

Making use of the mentioned Galerkin orthogonality, namely $$\masterB\big([u-u_h,p-p_h,\lambda-\lambda_h], [v_h,q_h,\mu_h]\big) = 0$$ for all $[v_h,q_h,\mu_h]\in\Xh$, we can add arbitrary discrete test functions. Hence, including $v_h \in \Vh$ as any quasi-interpolant of $v$ (such as the Cl{\'e}ment interpolant \cite{Cle75}) and $q_h \in \Qh$ as the (pointwise) interpolant of $q$, we can estimate 
\begin{align*}
&\masterB\big([u-u_h,p-p_h,\lambda-\lambda_h],[v,q,\mu]\big)\\
&=\masterB\big([u-u_h,p-p_h,\lambda-\lambda_h],[v-v_h,q-q_h,\mu]\big)\\
&\le \sum_{T \in \calT_\Omega} \big\|f - \sigma u_h + \Delta_\alpha u_h \big\|_{L^2(T)} \|v-v_h\|_{L^2(T)} 
+ \sum_{E \in \Ein} \big\| [ \alpha \nabla u_h \cdot n_E]_E \big\|_{L^2(E)} \|v-v_h\|_{L^2(E)}\\*
& \quad + \sum_{E \in \Ebd} \big\|\lambda_h - \alpha\nabla u_h \cdot n_E \big\|_{L^2(E)} \|v-v_h\|_{L^2(E)}\\*
& \quad + \sum_{I \in \calT_\Gamma} \big\|g - \sigma p_h + \Delta_{\Gamma,\kappa} p_h - \lambda_h\|_{L^2(I)} \|q - q_h\|_{L^2(I)}\\*
& \quad + \sum_{x\in \calN_\Gamma} \big\vert [\kappa\nabla_\Gamma p_h]_x \big\vert\,\underbrace{\vert q(x)-q_h(x)\vert}_{=0} 
+ \|u_h - p_h\|_{H^{\sfrac 12}(\Gamma)}\|\mu\|_{H^{-\sfrac 12}(\Gamma)} \\
&\lesssim \sum_{T \in \calT_\Omega} h_T\, \big\|f - \sigma u_h + \Delta_\alpha u_h \big\|_{L^2(T)} \|v\|_{H^1(\omega_T)} 
+ \sum_{E \in \Ein} h_E^{\sfrac 12}\, \big\| [\alpha \nabla u_h \cdot n_E]_E \big\|_{L^2(E)} \|v\|_{H^1(\omega_E)}\\*
& \quad + \sum_{E \in \Ebd} h_E^{\sfrac 12}\, \big\|\lambda_h - \alpha\nabla u_h \cdot n_E \big\|_{L^2(E)} \|v\|_{H^1(\omega_E)} 
\\*
& \quad+ \sum_{I \in \calT_\Gamma} h_I\, \big\|g - \sigma p_h + \Delta_{\Gamma,\kappa} p_h - \lambda_h\|_{L^2(I)} \|q\|_{H^1(I)} + \|u_h - p_h\|_{H^{\sfrac 12}(\Gamma)}\|\mu\|_{H^{-\sfrac 12}(\Gamma)} \\
&\lesssim \bigg( \sum_{T \in \calT_\Omega} h^2_T\, \big\|f - \sigma u_h + \Delta_\alpha u_h \big\|_{L^2(T)}^2 + \sum_{E \in \Ein} h_E\, \big\| [\alpha \nabla u_h \cdot n_E]_E \big\|^2_{L^2(E)} \\*
& \qquad+ \sum_{E \in \Ebd} h_E\, \big\|\lambda_h - \alpha\nabla u_h \cdot n_E \big\|_{L^2(E)}^2 + \sum_{I \in \calT_\Gamma} h_I^2\, \big\|g - \sigma p_h + \Delta_{\Gamma,\kappa} p_h - \lambda_h\|_{L^2(I)}^2 \\*
& \qquad+ \|u_h - p_h\|_{H^{\sfrac 12}(\Gamma)}^2 \bigg)^{1/2}\ \big\|[v,q,\mu]\big\|.
\end{align*}
Here, we have used standard (quasi) interpolation results as presented, e.g., in~\cite{Car06}. This includes the well-known element and edge patches, denoted by $\omega_T$ and $\omega_E$, respectively. 
Further note that the difference~$\tr u_h - p_h$ occurs globally, since the $H^{\sfrac 1 2}(\Gamma)$-norm is not additive. For a local version, we use that $\calT_{\Gamma}$ is a refinement of $\calT_{\Omega}|_\Gamma$. Thus, we can apply the interpolation inequality for $H^{1/2}(\Gamma)$, see \cite[Ch.1 Prop.~2.3]{LioM72} and \cite[p.~250 ff.]{AdaF03}, and the inverse estimate \cite[Ch.~II.6.8]{Bra07}. Together with the Poincar{\'e} inequality, 
this leads to 
\begin{align}
\label{eqn:est_constraint}
	\|u_h - p_h\|_{H^{\sfrac 12}(\Gamma)}^2  
	&\lesssim \|u_h - p_h\|_{H^{1}(\Gamma)}\cdot \|u_h - p_h\|_{L^{2}(\Gamma)}\\
	&\lesssim \vert u_h - p_h\vert_{H^{1}(\Gamma)}\cdot \|u_h - p_h\|_{L^{2}(\Gamma)} \notag\\
	&\lesssim \Big(\sum_{I \in \calT_\Gamma} h_I^{-2}\, \|u_h - p_h\|_{L^{2}(I)}^2 \cdot \sum_{I \in \calT_\Gamma} \|u_h - p_h\|_{L^{2}(I)}^2\Big)^{1/2} \notag\\
	&\le \bigg( \frac{h_{\Gamma,\max}}{h_{\Gamma,\min}}\bigg)^{\sfrac 12} \sum_{I \in \calT_\Gamma} h_I^{-1} \|u_h - p_h\|_{L^{2}(I)}^2 \notag
\end{align}
with a hidden constant only depending on~$\Gamma$ and the polynomial degrees of~$\Vh$ and~$\Qh$. Here, $h_{\Gamma,\max}$ (respectively $h_{\Gamma,\min}$) denotes the maximal (minimal) mesh width of the triangulation~$\calT_\Gamma$; see also Remark~\ref{rem:est_constraint} below. 

Finally, we define the following local quantities, which will serve as local error estimators, namely 
\begin{subequations}
\label{eqn:etas}
\begin{align}
	\eta_T^2 
	&\coloneqq h_T^2\, \big\|f - \sigma u_h + \Delta_\alpha u_h \big\|_{L^2(T)}^2
	\hspace{5.35cm}\text{for }T \in \calT_\Omega, \\
	\eta_E^2 
	&\coloneqq h_E\, \big\| [\alpha \nabla u_h \cdot n_E]_E \big\|_{L^2(E)}^2 
	\hspace{5.85cm}\text{for } E \in \Ein, \\
	\eta_E^2 
	&\coloneqq h_E\, \big\|\lambda_h - \alpha\nabla u_h \cdot n_E \big\|_{L^2(E)}^2 + \sum_{\substack{I \in \calT_\Gamma, I \subseteq E}} h_I^{-1}\, \big\|u_h - p_h\big\|_{L^{2}(I)}^2 
	\quad\text{for } E \in \Ebd, \\
	\eta_I^2 
	&\coloneqq h_I^2\, \big\|g - \sigma p_h + \Delta_{\Gamma,\kappa} p_h - \lambda_h \big\|_{L^2(I)}^2
	\hspace{4.43cm}\text{for }I \in \calT_{\Gamma}.  
\end{align}
\end{subequations}
The remainder of this section is devoted to the verification of the reliability and efficiency of these estimators. 
\begin{remark}
\label{rem:3d}
In the above derivation of the error estimators, we have used that functions in~$H^1(\Gamma)$ are continuous for a one-dimensional boundary~$\Gamma$ of a two-dimensional domain~$\Omega$. This allows us to neglect the jumps of $p$ over the nodes of $\calT_\Gamma$. If $\Omega$ is a subset of $\R^3$, however, we have to consider these jumps, which are now jumps over edges. As a consequence, the estimators in~\eqref{eqn:etas} need to be extended by 
\begin{equation}
\label{eqn:add_eta_3D}
	\eta_{\eg}^2 
	\coloneqq h_{\eg}\, \big\|[\kappa \nabla_\Gamma p_h \cdot n_\eg]_\eg\big\|_{L^2(\eg)}^2 \qquad \text{for } \eg \in \calE_\Gamma
\end{equation}
with $\calE_\Gamma$ being the set of edges of~$\calT_\Gamma$ and $h_{\eg}$ the length of an edge~$\eg \in \calE_\Gamma$. 
\end{remark}
%
%
\subsection{Reliability and efficiency}
\label{sect:aposteriori:thm}

The reliability of the introduced a posteriori estimators follows directly by construction and is summarized in the following theorem. 
\begin{theorem}[Reliability of the error estimator]
\label{th:reliability}
Let $\eta_T$, $\eta_E$, and $\eta_I$ be defined as in~\eqref{eqn:etas} and let Assumption~\ref{ass:disc} be satisfied. Then the discretization error can be bounded by
\begin{align*}
	\big\|[u-u_h, p-p_h, \lambda-\lambda_h] \big\|^2
	&= \|u-u_h\|_{H^1(\Omega)}^2 + \|p-p_h\|_{H^1(\Gamma)}^2 + \|\lambda-\lambda_h\|_{H^{-\sfrac 12}(\Gamma)}^2 \\
	&\lesssim \sum_{T \in \calT_\Omega} \eta_T^2 + \sum_{E \in \calE_\Omega} \eta_E^2 + \sum_{I \in \calT_{\Gamma} } \eta_I^2,
\end{align*}
where the hidden constant only depends on the dimension $d$, the polynomial degrees used for the finite element spaces, the polygonal domain~$\Omega$ with its boundary~$\Gamma$, and the shape regularity of the triangulations $\calT_\Omega$ and $\calT_\Gamma$, and the coefficients~$\alpha$, $\kappa$, and $\sigma$. 
\end{theorem}
\begin{remark}\label{rem:est_constraint}
The calculation~\eqref{eqn:est_constraint} shows that the hidden constant contains the ratio $h_{\Gamma,\max} / h_{\Gamma,\min}$, which may be large due to adaptive refinements. 
An alternative estimate of $u_h-p_h$ on $\Gamma$ allows to replace this ratio by a constant which is independent of the mesh widths of~$\calT_\Gamma$ and~$\calT_\Omega$; cf.~\cite[Lem.~2.4.6]{Lip04}. For this, assume the use of bisection for $d=2$ or newest vertex bisection for $d=3$ to obtain~$\calT_\Gamma$. Then, there exists a refinement~$\widetilde{\calT}_{\Omega}$ of~$\calT_{\Omega}$ such that $\calT_\Gamma = \widetilde{\calT}_{\Omega}|_\Gamma$. Moreover, the shape regularity of $\widetilde{\calT}_\Omega$ only depends on~$\calT_\Omega$ and the dimension~$d$; see~\cite[Sect.~A.1]{DemS11}. 
\end{remark}
\begin{remark}
The reliability in the three-dimensional case follows similarly to Theorem~\ref{th:reliability}, where we have the additional term~$\sum_{\eg \in \calE_\Gamma} \eta_\eg^2$ on the right-hand side, cf.~Remark~\ref{rem:3d}.
\end{remark}
Theorem~\ref{th:reliability} guarantees that the error estimators provide an upper bound of the error. It remains to prove the efficiency of the estimators, i.e., the guarantee that they are (up to a constant) also a lower bound. 

In the following lemmas, we always assume that Assumption~\ref{ass:disc} is satisfied. Similar to Theorem~\ref{th:reliability}, all hidden constants will only depend on the dimension, the polynomial degrees of the finite element spaces, the domain with its boundary, the shape regularity of the meshes, the diffusion coefficients, and the constant $\rho$. We first consider the estimators~$\eta_T$ and $\eta_I$. 
\begin{lemma}\label{lem:eta_T_and_eta_I}
The estimators $\eta_T$ and $\eta_I$ 
satisfy
\begin{subequations}
\begin{align}
	\eta_T^2 
	\lesssim \sigma^2 h_T^2\, \|u-u_h\|_{L^2(T)}^2 + \|\alpha\|_{L^{\infty}(T)}^2\vert u-u_h\vert_{H^1(T)}^2 + 
	h_T^2\inf_{f_h \in \Vh} \|f-f_h\|_{L^2(T)}^2 
	\label{eqn:eta_T_estimate}
\end{align}
as well as
\begin{align}
	\eta_I^2
	\lesssim \sigma^2 h_I^2\, \|p-p_h\|_{L^2(I)}^2 + \|\kappa\|_{L^{\infty}(I)}^2\vert p-p_h\vert_{H^1(I)}^2 + h_I\, \|\lambda-\lambda_h\|_{H^{-\sfrac 12}(I)}^2 + h_I^2\inf_{g_h \in \Qh} \|g-g_h\|_{L^2(I)}^2. 
	\label{eqn:eta_I_estimate}	
\end{align}
\end{subequations}
These estimates further imply 
\begin{subequations}
	\label{eqn:eta_estimate_sum}
	\begin{gather}
	\sum_{T \in \calT_\Omega} \eta_T^2  \lesssim \|u-u_h\|^2_{H^1(\Omega)}  + \sum_{T \in \calT_\Omega} 
	h_T^2 \inf_{f_h \in \Vh} \|f-f_h\|^2_{L^2(T)},\label{eqn:eta_T_estimate_sum}\\
	\sum_{I \in \calT_{\Gamma}} \eta_I^2 
	\lesssim  \|p-p_h\|^2_{H^1(\Gamma)} + h_{\Gamma,\max}\, \|\lambda-\lambda_h\|_{H^{-\sfrac 12}(\Gamma)}^2 + \sum_{I \in \calT_{\Gamma}} h_I^2 \inf_{g_h \in \Qh} \|g-g_h\|^2_{L^2(I)} \label{eqn:eta_I_estimate_sum}.
	\end{gather}
\end{subequations}
\end{lemma}
\begin{proof}
Estimate~\eqref{eqn:eta_T_estimate} follows by the lines of \cite[Lem.~2.6.10(1)]{Lip04}. For~\eqref{eqn:eta_I_estimate}, we note that every segment~$I \in \calT_\Gamma$ is a part of a hyperplane $\mathbb{H}_I \cong \R^{d-1}$. With this, one proves~\eqref{eqn:eta_I_estimate} analogously to~\eqref{eqn:eta_T_estimate}, where one uses that for every $\nu \in H^1_0(I)$ -- extended by zero outside of $I$ -- it holds the interpolation inequality 
\begin{align*}
	\|\nu\|_{H^{\sfrac 12}(I)}^2 
	&= \int_I \nu(x)^2 \dx + \int_I \int_I \frac{\vert \nu(x)-\nu(y) \vert^2}{\|x-y\|_{\R^d}^d} \dx \dy \\
	&\leq \int_{\mathbb{H}_I} \nu(x)^2 \dx + \int_{\mathbb{H}_I} \int_{\mathbb{H}_I} \frac{\vert \nu(x)-\nu(y) \vert^2}{\|x-y\|_{\R^d}^d} \dx \dy \\
	&= \|\nu\|_{H^{\sfrac 12}(\R^{d-1})}^2 
	\lesssim \|\nu \|_{H^1(\R^{d-1})}\|\nu \|_{L^2(\R^{d-1})} 
	= \|\nu \|_{H^1(I)}\|\nu \|_{L^2(I)}, 
\end{align*}
see~\cite[Ch.1 Prop.~2.3 \& Th.~7.1]{LioM72}. 
%
Inequality~\eqref{eqn:eta_T_estimate_sum} is a immediate consequence of~\eqref{eqn:eta_T_estimate} together with $h_T \leq |\Omega|$. Finally, \eqref{eqn:eta_I_estimate_sum} can be shown in the same manner as~\eqref{eqn:eta_I_estimate} if adjusted to the entire boundary $\Gamma$.
\end{proof}
To obtain an upper bound on $\eta_E$, we distinguish interior and boundary edges. 
\begin{lemma}\label{lem:eta_E_int}
For an interior edge $E \in \Ein$ it holds that
\[
	\eta_E^2 
	\lesssim \sigma^2 h_E^2\, \|u-u_h\|_{L^2(\omega_E)} + \|\alpha\|^2_{L^\infty(\omega_E)}|u-u_h|_{H^1(\omega_E)}^2 + \sum_{\substack{T\in \calT_\Omega,\ T \subseteq \omega_E}} \eta_T^2
\]
and, hence, 
\[	
	\sum_{E \in \Ein} \eta_E^2 \lesssim  \|u-u_h\|_{H^1(\Omega)}^2 + \sum_{T \in \calT_\Omega} 
	h_T^2 \inf_{f_h \in \Vh} \|f-f_h\|^2_{L^2(T)}.
\]
\end{lemma}
\begin{proof}
The proof follows the lines of \cite[Lem.~3.6]{AinT97}.
\end{proof}
\begin{lemma}\label{lem:eta_E_bd}
For a boundary edge $E \in \Ebd$ we define the boundary edge patch~$\omega_{E,\Gamma}$ by 
\begin{equation*}
	\omega_{E,\Gamma} 
	\coloneqq \big\{ E^\prime \in \Ebd\ |\ \overline{E}\cap \overline{E^\prime} \neq \emptyset \big\}
\end{equation*}
and $T_E$ as the (unique) element in $\calT_\Omega$ with edge $E$. Then, it holds that 
\begin{equation*}
	\eta_E^2 
	\lesssim \| u-u_h\|_{H^1(\omega_{T_E})}^2 + 
	\|p-p_h\|_{H^{1}(\omega_{E,\Gamma})}^2 
	+ h_E^{\sfrac 12}\, \|\lambda - \lambda_{h}\|_{H^{-\sfrac 12}(E)}^2 + h_{T_E}^2\inf_{f_h \in \Vh} \|f-f_h\|_{L^2(T_E)}^2 
\end{equation*}
and, hence, 
%
\begin{equation*}
	\sum_{E \in \Ebd} \eta_E^2 \lesssim \| u-u_h\|_{H^1(\Omega)}^2 + 
	\|p-p_h\|_{H^{1}(\Gamma)}^2 
	+ \|\lambda - \lambda_{h}\|_{H^{-\sfrac 12}(\Gamma)}^2 + \sum_{E \in \Ebd} h_{T_E}^2\inf_{f_h \in \Vh} \|f-f_h\|_{L^2(T_E)}^2.
\end{equation*}
\end{lemma}
\begin{proof}
Combining the arguments of \cite[Lem.~2.6.12]{Lip04}, \cite[Lem.~3.6]{AinT97}, and the one used for the derivation of~\eqref{eqn:eta_I_estimate}, we have
\begin{multline*}
	h_E\, \|\lambda_h - \alpha \nabla u_h \cdot n_E\|_{L^2(E)}^2 
	\lesssim \sigma^2 h_E^2\, \|u-u_h\|_{L^2(T_E)} + \|\alpha\|^2_{L^\infty(T_E)}\vert u-u_h\vert_{H^1(T_E)}^2\\ + h_E^{\sfrac 12}\, \|\lambda - \lambda_{h}\|_{H^{-\sfrac 12}(E)}^2 + \eta_{T_E}^2 
\end{multline*}
and, hence, for the entire boundary
\begin{equation*}
	\sum_{E \in \Ebd} h_E\, \|\lambda_h - \alpha \nabla u_h \cdot n_E\|_{L^2(E)}^2 \lesssim \| u-u_h\|_{H^1(\Omega)}^2 + \|\lambda - \lambda_{h}\|_{H^{-\sfrac 12}(\Gamma)}^2 + \sum_{E \in \Ebd} \eta_{T_E}^2.
\end{equation*}
For an estimate of the second summand of $\eta_E$, we set $\mu \coloneqq \sum_{I \in \calT_\Gamma, I \subseteq E} h_I^{-1} (u_h - p_h)\, \chi_I \in L^2(\Gamma)\hookrightarrow H^{-\sfrac 12}(\Gamma)$ with the characteristic function $\chi_I$. We now distinguish the two cases~$\calP^\text{d}_0(\calT_{\Omega}|_{\Gamma})\subseteq \Mh$ and~$\calP_1(\calT_{\Omega}|_{\Gamma}) \subseteq \Mh$, cf.~Assumption~\ref{ass:disc}(1).
\medskip
	
\emph{Case $\calP_0^\text{d}(\calT_{\Omega}|_{\Gamma}) \subseteq \Mh$}: If piecewise constant functions are part of~$\Mh$, we define the mean integral of $\mu$ as $\mu_h \coloneqq \frac{1}{|E|}\int_E \mu \dx\cdot \chi_E \in \Mh$. Then, the continuous and discrete constraint imply that 
\begin{align*}
	\sum_{I \in \calT_\Gamma,\ I \subseteq E} h_I^{-1} \|u_h - p_h\|_{L^{2}(I)}^2 
	&= \int_{E} \mu\,(u_h-p_h) \dx\\
	&= \int_{E} (\mu-\mu_h)\big(u_h-u-(p_h-p)\big) \dx \\
	&\leq \|\mu - \mu_h\|_{H^{-\sfrac 12}(E)} \big(\|u-u_h\|_{H^{\sfrac 12}(E)} + \|p-p_h\|_{H^{\sfrac 12}(E)}\big).
\end{align*}
By \cite[Prop.~2.2 \&~4.8]{AcoB17}, there exists a constant only depending on the shape regularity of $\calT_\Omega|_{\Gamma}$ such that
\begin{align*}
	\|\mu-\mu_h\|_{H^{-\sfrac 1 2}(E)}\, 
	&= \sup_{q \in H^{\sfrac{1}{2}}(E)\setminus\{0\}} \frac{(\mu-\mu_h,q)_{L^2(E)}}{\|q\|_{H^{\sfrac 1 2} (E)}}\\
	&= \sup_{q \in H^{\sfrac 1 2}(E)\setminus\{0\}} \frac{(\mu,q-\tfrac{1}{|E|}\int_E q \dx)_{L^2(E)}}{\|q\|_{H^{\sfrac{1}{2}}(E)}}\\
	&\lesssim h^{\sfrac 1 2}_E\, \|\mu\|_{L^{2}(E)}
	\leq \sqrt \rho\, \Big( \sum_{I \in \calT_\Gamma, I \subseteq E} h_I^{-1} \|u_h - p_h\|_{L^{2}(I)}^2\Big)^{\sfrac 1 2}.
\end{align*}
Combining the latter two estimates, we conclude that 
\begin{equation*}
	\sum_{\substack{I \in \calT_\Gamma, I \subseteq E}} h_I^{-1} \|u_h - p_h\|_{L^{2}(I)}^2 
	\lesssim \|u-u_h\|_{H^{\sfrac 12}(E)}^2 + \|p-p_h\|_{H^{\sfrac 12}(E)}^2.
\end{equation*}
With Lemma~\ref{lem:eta_T_and_eta_I} and appropriate Sobolev embeddings, we get the stated estimate for $\eta_E$. Finally, the estimate of the sum follows by $\sum_{E \in \Ebd} \|\cdot \|_{H^{\sfrac 12}(E)}^2 \leq \|\cdot \|_{H^{\sfrac 12}(\Gamma)}^2$. 
\medskip
	
\emph{Case $\calP_1(\calT_{\Omega}|_{\Gamma}) \subseteq \Mh$}: 
Let $P_E\colon L^2(\omega_{E,\Gamma}) \to \calP_1(\calT_\Omega|_{\omega_{E,\Gamma}}) \subseteq H^1(\omega_{E,\Gamma})$ be the (weighted) Clément operator 
\begin{equation*}
	P_E\,r 
	= \sum_{k=1}^m \frac{(r,\phi_k)_{L^2(\omega_{E,\Gamma})}}{(1,\phi_k)_{L^2(\omega_{E,\Gamma})}}\, \phi_k
\end{equation*}
with the nodal basis functions~$\phi_k$, $k=1,\ldots,m$, of~$\calP_1(\calT_\Omega|_{\omega_{E,\Gamma}})$. Since $\mu$ vanishes outside of~$E$, its quasi-interpolation~$P_E \mu$ is zero on the relative boundary $\partial \omega_{E,\Gamma}$. Hence, $P_E \mu$ can be extended by zero to a function in $\Mh$. Similar to the previous case, we obtain the estimate 
\begin{equation*}
	\sum_{I \in \calT_\Gamma, I \subseteq E} h_I^{-1}\, \|u_h - p_h\|_{L^{2}(I)}^2 
	 \leq \|\mu - P_E \mu\|_{H^{-\sfrac 12}(\omega_{E,\Gamma})} \big(\|u-u_h\|_{H^{\sfrac 12}(\omega_{E,\Gamma})} + \|p-p_h\|_{H^{\sfrac 12}(\omega_{E,\Gamma})}\big).
\end{equation*}
By \cite[Lem.~3.1]{BraPV00}, we have $\|\mu - P_E \mu\|_{L^2(\omega_{E,\Gamma})} \lesssim \|\mu\|_{L^2(E)}$ as well as
\begin{align*}
	\|\mu-P_E \mu\|_{[H^{1}(\omega_{E,\Gamma})]^\ast}\, 
	= \sup_{q \in H^{1}(E)\setminus\{0\}} \frac{(\mu,q-P_E q)_{L^2(E)}}{\|q\|_{H^{1} (\omega_{E,\Gamma})}} 
	\lesssim h_E\, \|\mu\|_{L^2(E)}
\end{align*}
with constants only depending on the regularity of~$\calT_\Omega|_\Gamma$. Here, we have used that~$P_E$ is symmetric with respect to the $L^2(\omega_{E,\Gamma})$-inner product. With the interpolation inequality we conclude $\|\mu - P_E \mu\|_{H^{-1/2}(\omega_{E,\Gamma})} \lesssim h_E^{\sfrac{1}{2}}\, \|\mu\|_{L^{2}(E)}$.
The remainder of the proof follows the lines of the first case.
\end{proof}
It remains to summarize the previous estimates. The combination of Lemmas~\ref{lem:eta_T_and_eta_I}, \ref{lem:eta_E_int}, and~\ref{lem:eta_E_bd} yield the following efficiency result, showing that the error estimators are bounded from above by the actual error plus oscillation terms of the right-hand sides.
\begin{theorem}[Efficiency of the error estimator]
\label{th:efficiency}
Given Assumption~\ref{ass:disc}, the error estimators defined in~\eqref{eqn:etas} satisfy 
\begin{align*}
	&\quad\sum_{T \in \calT_\Omega} \eta_T^2 + \sum_{E \in \calE_\Omega} \eta_E^2 + \sum_{I \in \calT_{\Gamma} } \eta_I^2\\
	&\lesssim \| u-u_h\|_{H^1(\Omega)}^2 + 
	\|p-p_h\|_{H^{1}(\Gamma)}^2 
	+ \|\lambda - \lambda_{h}\|_{H^{-\sfrac 12}(\Gamma)}^2 \\
	&\qquad+ \sum_{T \in \calT_\Omega} h_{T}^2\inf_{f_h \in \Vh} \|f-f_h\|_{L^2(T)}^2 + \sum_{I \in \calT_{\Gamma}} h_I^2 \inf_{g_h \in \Qh} \|g-g_h\|^2_{L^2(I)} \\
	&= \big\|[u-u_h,p-p_h,\lambda-\lambda_h]\big\|^2 + \sum_{T \in \calT_\Omega} h_{T}^2\inf_{f_h \in \Vh} \|f-f_h\|_{L^2(T)}^2 + \sum_{I \in \calT_{\Gamma}} h_I^2 \inf_{g_h \in \Qh} \|g-g_h\|^2_{L^2(I)}.
\end{align*}
The hidden constant only depends on the dimension $d$, the polynomial degrees of the finite element spaces, the polygonal domain~$\Omega$ with its boundary~$\Gamma$, the shape regularity of the triangulations~$\calT_\Omega$ and~$\calT_\Gamma$, the coefficients~$\alpha$ and $\kappa$, as well as the constants~$\sigma$ and $\rho$. 
\end{theorem}
\begin{remark}
Analogously to $\eta_E$ in Lemma~\ref{lem:eta_E_int}, the additional error estimator $\eta_{\eg}$ defined in~\eqref{eqn:add_eta_3D} is bounded by
\[
	\eta_\eg^2 
	\lesssim \sigma^2 h_\eg^2\, \|p-p_h\|_{L^2(\omega_\eg)} + \|\kappa\|^2_{L^\infty(\omega_\eg)}|p-p_h|_{H^1(\omega_\eg)}^2 + \sum_{\substack{I\in \calT_\Gamma, I \subseteq \omega_\eg}} \eta_I^2
\]
for a three-dimensional (polyhedral) domain $\Omega$. Here, $\omega_\eg$ is defined with respect to $\calT_{\Gamma}$ as $\omega_E$ to $\calT_\Omega$. Furthermore, Lemmas~\ref{lem:eta_T_and_eta_I}, \ref{lem:eta_E_int}, and~\ref{lem:eta_E_bd} are also valid for $\Omega\subseteq \R^3$, implicating again the efficiency of the error estimators. 
\end{remark}
%
%
\section{Numerical Experiments}\label{sect:numerics}

In this final section, we demonstrate the performance of the newly introduced a posteriori error estimators applied to the stationary model problem~\eqref{eqn:saddle_point_AFEM} as well as a parabolic problem with dynamic boundary conditions. 

In all examples, we consider a two-dimensional domain and diffusion coefficients~$\alpha\equiv1$, $\kappa\equiv1$. Moreover, we use bisection as refinement strategy for $\calT_\Gamma$ and newest vertex bisection for $\calT_\Omega$ with the implementation taken from~\cite{FunPW11}. Since newest vertex bisection is based on element marking (instead of marking elements and edges), we attribute the estimator~$\eta_E$ of an edge~$E$ among all~$\eta_T$ and~$\eta_{I}$ with $T \subseteq \omega_E$ and $I \subseteq E$, respectively. The resulting error estimators are denoted by $\widetilde{\eta}_T$ and $\widetilde{\eta}_I$. As marking strategy, we use the well-knwon {\em D\"orfler marking}~\cite{Dor96} with a parameter $\theta \in (0,1)$, i.e., we mark all elements in (minimal) sets $\calM_\Omega\subseteq\calT_\Omega$ and $\calM_\Gamma\subseteq\calT_\Gamma$ such that  
\begin{equation*}
	(1-\theta)\, \bigg(\sum_{T\in \calT_\Omega} \widetilde{\eta}_T^2 + \sum_{I\in \calT_\Gamma} \widetilde{\eta}_I^2 \bigg) 
	\leq \sum_{T\in \calM_\Omega} \widetilde{\eta}_T^2 + \sum_{I\in \calM_\Gamma} \widetilde{\eta}_I^2.
\end{equation*}
For the computation of the errors, we use a reference solution, which is computed on a refined mesh. More precisely, we consider in each step the current mesh obtained by the adaptive process with two additional uniform refinements. 
%
%
\subsection{Stationary problem on the unit square}
\label{sect:numerics:square}

In this first example, we consider the stationary model problem~\eqref{eqn:saddle_point_AFEM} on the unit square, i.e., on $\Omega = (0,1)^2$. The right-hand sides are given by 
\[
	\fu \equiv 0.04, \qquad 
	\fp(x,y) = xy \cos(10 \pi x) \cos(10 \pi y)
\]
and we set $\sigma = 1$ as well as $\theta = 0.75$ for the marking process. The resulting convergence history for the scheme 
\begin{equation}
\label{eqn:P1scheme}
	\Vh = \calP_1(\calT_{\Omega}), \qquad 
	\Qh = \calP_1(\calT_{\Gamma}), \qquad 
	\Mh = \calP_1(\calT_{\Omega}|_{\Gamma})
\end{equation}
is illustrated in Figure~\ref{fig:stationary_square}. Note that the $x$-axis includes the sum of the degrees of freedom for $u$, $p$, and $\lambda$. The plot clearly shows the improvement caused by the adaptive process as we reach a convergence rate of $0.65$ rather than $0.5$ for a uniform refinement. Moreover, one can observe that error and estimator, which are defined as 
\[
	\text{error} 
	\coloneqq \big\| [u-u_h,p-p_h,\lambda-\lambda_h] \big\|, \quad 
	\text{estimator} 
	\coloneqq \bigg(\sum_{T\in \calT_\Omega} \eta_T^2 + \sum_{E\in \calE_\Omega} \eta_E^2 + \sum_{I\in \calT_\Gamma} \eta_I^2 \bigg)^{1/2} 
\] 
are indeed of the same order as expected due to Theorems~\ref{th:reliability} and~\ref{th:efficiency}. 
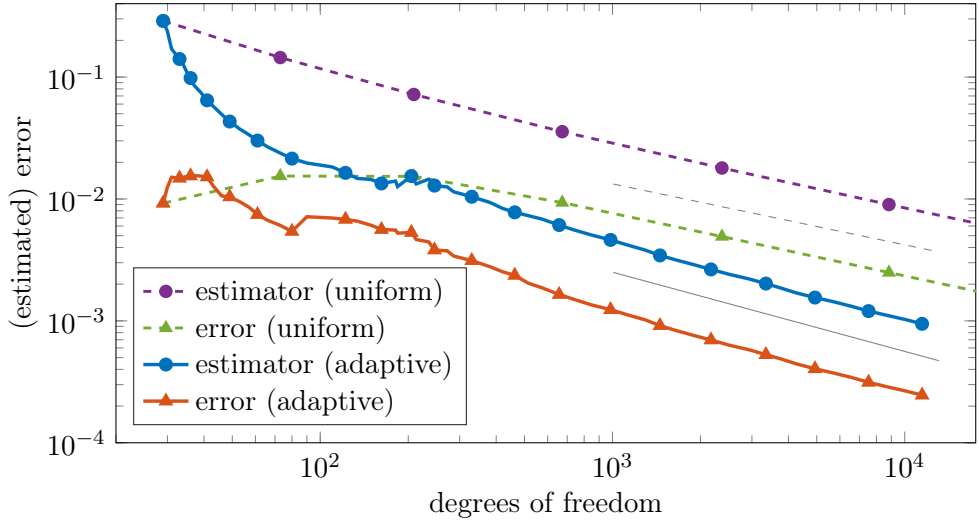
\begin{figure}
%
%
\definecolor{mycolor1}{rgb}{0.00000,0.44700,0.74100}%
\definecolor{mycolor2}{rgb}{0.85000,0.32500,0.09800}%
\definecolor{mycolor3}{rgb}{0.49400,0.18400,0.55600}%
\definecolor{mycolor4}{rgb}{0.46600,0.67400,0.18800}%
\begin{tikzpicture}

\begin{axis}[%
width=4.5in,
height=2.3in,
at={(2.6in,1.27in)},
scale only axis,
xmode=log,
xmin=20,
xmax=1.75e4,
xminorticks=true,
xlabel={degrees of freedom},
ymode=log,
ymin=1e-04,
ymax=0.4,
yminorticks=true,
ylabel={(estimated) error},
axis background/.style={fill=white},
legend style={legend cell align=left, align=left, at={(0.02,0.04)}, anchor=south west, draw=white!15!black}
]
\addplot [color=mycolor3, line width=1.1pt, dashed, mark=*, mark options={solid, mycolor3}]
table[row sep=crcr]{%
29	0.289077178733986\\
73	0.144543130621917\\
209	0.0719917453131967\\
673	0.0356435238693348\\
2369	0.0179479307845777\\
8833	0.00899918504965468\\
34049	0.0045040934089755\\
};
\addlegendentry{estimator (uniform)}

\addplot [color=mycolor4, line width=1.1pt, dashed, mark=triangle*, mark options={solid, mycolor4}]
table[row sep=crcr]{%
29	0.00922824163085455\\
73	0.015419539515169\\
209	0.015333530953737\\
673	0.0093382004753873\\
2369	0.00491772624702099\\
8833	0.00249099810787252\\
34049	0.00124951175828402\\
};
\addlegendentry{error (uniform)}

\addplot [color=gray, dashed, forget plot]
table[row sep=crcr]{%
	1006	0.013261909638069\\
	13087	0.003676926965610\\	
};

\addplot [color=mycolor1, line width=1.3pt, mark=*, mark options={solid, mycolor1}, mark repeat=3]
  table[row sep=crcr]{%
29	0.289077178733986\\
30	0.237106749804394\\
31	0.169921143639504\\
33	0.14069634631624\\
34	0.125973951766013\\
35	0.109285885204163\\
36	0.0981518259323994\\
37	0.085585535150427\\
39	0.0748002522395517\\
41	0.0644563844850817\\
43	0.0577002514563067\\
46	0.0494517187937144\\
49	0.0431769146318515\\
53	0.0375333092799185\\
57	0.0336754840424281\\
61	0.0301292596190646\\
66	0.0267867220957189\\
73	0.0237923169447894\\
80	0.0214411628044468\\
90	0.0197242973873901\\
110	0.0183101501805041\\
122	0.0163767087447767\\
134	0.0147303947925865\\
149	0.0144740624296959\\
162	0.0134061042529231\\
181	0.0140208618451674\\
184	0.0126297299150079\\
205	0.0154087077110136\\
213	0.0132837177423147\\
235	0.0145268383123242\\
246	0.0128711656842504\\
272	0.012547168097245\\
286	0.011446321964908\\
330	0.0104266414460085\\
381	0.00936103123651732\\
420	0.00837733977601331\\
463	0.00777149845536296\\
516	0.00728805098345933\\
573	0.00684456856881421\\
657	0.00610973181349622\\
736	0.00558547693502258\\
846	0.00506925517413575\\
985	0.00461343003071105\\
1140	0.0041531737586773\\
1295	0.00380462404206552\\
1454	0.00344421195710975\\
1659	0.00314822980437839\\
1914	0.00287565562022366\\
2173	0.00264166148689119\\
2508	0.00240990231805243\\
2920	0.00221338778525475\\
3350	0.00202236463836094\\
3783	0.00185322465004344\\
4263	0.00169397021470252\\
4930	0.00155367281649747\\
5732	0.00143160138493401\\
6585	0.00131512429441096\\
7522	0.00120275581758005\\
8698	0.00110976296702797\\
10008	0.00102781223865196\\
11465	0.000945474964968424\\
};
\addlegendentry{estimator (adaptive)}

\addplot [color=mycolor2, line width=1.3pt, mark=triangle*, mark options={solid, mycolor2}, mark repeat=3]
  table[row sep=crcr]{%
29	0.00922824163085455\\
30	0.0125005082245405\\
31	0.0150774384654501\\
33	0.0147884130802104\\
34	0.0150593189575008\\
35	0.0153254658450376\\
36	0.0154405217492682\\
37	0.0155538252810128\\
39	0.0154157668836151\\
41	0.0151590641043682\\
43	0.0128078262432804\\
46	0.0105833383801445\\
49	0.0104085599593353\\
53	0.00949385396952551\\
57	0.00857276046798622\\
61	0.0074513135829646\\
66	0.00669018223154488\\
73	0.00607044656957036\\
80	0.00541773878948118\\
90	0.00715050031781529\\
110	0.0069749421891312\\
122	0.00680203506521164\\
134	0.00657161162202037\\
149	0.00609909839456634\\
162	0.00561226336566203\\
181	0.005559601748674\\
184	0.00524316964247748\\
205	0.00530986387185974\\
213	0.00464618833299466\\
235	0.00444508114968729\\
246	0.00381591499639849\\
272	0.00375616001797948\\
286	0.00340335074018526\\
330	0.00312252966033923\\
381	0.00278054978859932\\
420	0.00253212092355556\\
463	0.00236104283904284\\
516	0.00203307024553635\\
573	0.0018708620030838\\
657	0.00164491696798884\\
736	0.00150480078478236\\
846	0.0013537222951547\\
985	0.00123586598386637\\
1140	0.00111012027963332\\
1295	0.00101895763184055\\
1454	0.000913877671207956\\
1659	0.000830150943845939\\
1914	0.000755388989105512\\
2173	0.000697354215754946\\
2508	0.000631798407417221\\
2920	0.000581262955372462\\
3350	0.000528514131236154\\
3783	0.000485874267564439\\
4263	0.00044367080038134\\
4930	0.000404676414455984\\
5732	0.000372482839490197\\
6585	0.000343170860270774\\
7522	0.000313235028034667\\
8698	0.000287967193016109\\
10008	0.000266909984225669\\
11465	0.000245982264024888\\
};
\addlegendentry{error (adaptive)}

\addplot [color=gray, forget plot]
  table[row sep=crcr]{%
1006	0.0024991729678321\\
13087	0.000471562146181624\\
};
\end{axis}
\end{tikzpicture}%
	\caption{Convergence history for the stationary model problem of Section~\ref{sect:numerics:square}. The dashed line in gray indicates order~$0.5$, whereas the solid gray line indicates order~$0.65$. }
	\label{fig:stationary_square}
\end{figure}

The resulting mesh of the adaptive process is shown in Figure~\ref{fig:mesh_square}. It clearly shows a refinement of the upper right corner for $\calT_\Omega$ as well as for $\calT_\Gamma$. This is due to the fact that the right-hand side~$\fp$ as well as the solution oscillate strongly in this region. 
\begin{figure}
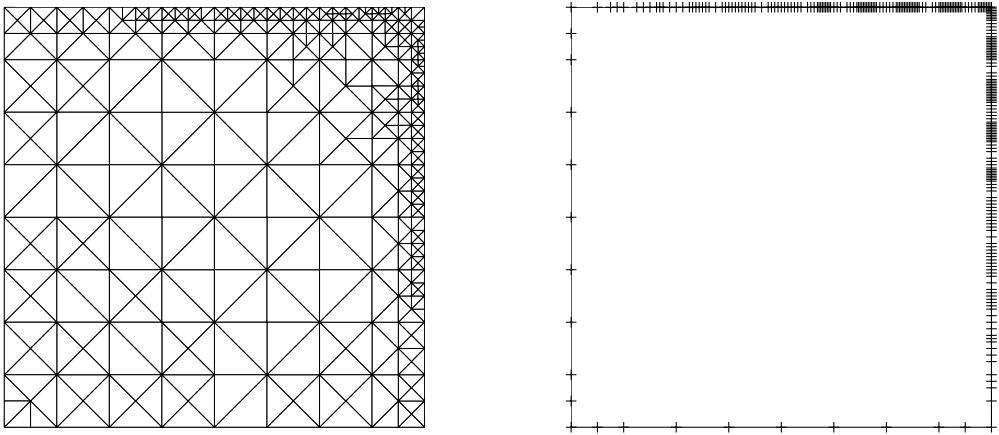

	\centering
	\include{pics/mesh_square}
	\caption{Illustration of the bulk mesh $\calT_\Omega$ (left) and the boundary mesh $\calT_\Gamma$ (right) after 40 adaptive refinement steps.}
	\label{fig:mesh_square}
\end{figure}
%
%
\subsection{Stationary problem on the L-shape}
\label{sect:numerics:Lshape}

The second example works on the L-shape, i.e., on a non-convex computational domain. Here, we consider the right-hand sides
\[
	\fu \equiv 4, \qquad 
	\fp(x,y) = 4\, (x^2-x + y^2 - y)
\]
and again $\sigma = 1$, $\theta = 0.75$. We compare two different stable finite element schemes. 
%
%

First, we use piecewise linear elements as in~\eqref{eqn:P1scheme}. As shown in Figure~\ref{fig:stationary_Lshape}, this yields the rate~$0.375$ for a uniform refinement and the optimal rate~$0.5$ in the adaptive case. 
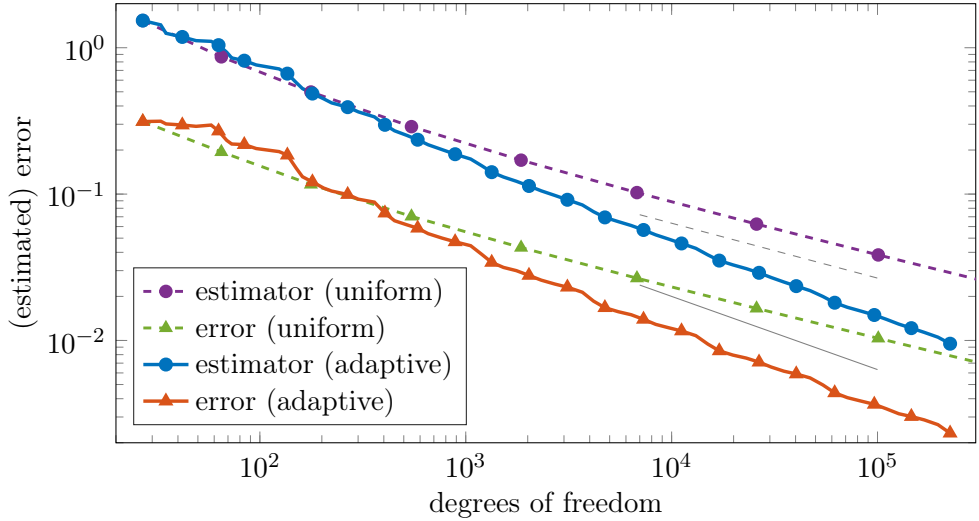
\begin{figure}
%
%
\definecolor{mycolor1}{rgb}{0.00000,0.44700,0.74100}%
\definecolor{mycolor2}{rgb}{0.85000,0.32500,0.09800}%
\definecolor{mycolor3}{rgb}{0.49400,0.18400,0.55600}%
\definecolor{mycolor4}{rgb}{0.46600,0.67400,0.18800}%
\begin{tikzpicture}

\begin{axis}[%
width=4.5in,
height=2.3in,
at={(2.6in,1.27in)},
scale only axis,
xmode=log,
xmin=20,
xmax=3e5,
xminorticks=true,
xlabel={degrees of freedom},
ymode=log,
ymin=0.002,
ymax=2,
yminorticks=true,
ylabel={(estimated) error},
axis background/.style={fill=white},
legend style={legend cell align=left, align=left, at={(0.02,0.04)}, anchor=south west, draw=white!15!black}
]
\addplot [color=mycolor3, line width=1.1pt, dashed, mark=*, mark options={solid, mycolor3}]
table[row sep=crcr]{%
	27	1.53306021270245\\
	65	0.866394489180529\\
	177	0.498361519023667\\
	545	0.288994127408124\\
	1857	0.170529394879733\\
	6785	0.102388550167333\\
	25857	0.0623651025633234\\
	100865	0.0384010980432459\\
	398337	0.0238286870704854\\
};
\addlegendentry{estimator (uniform)}

\addplot [color=mycolor4, line width=1.1pt, dashed, mark=triangle*, mark options={solid, mycolor4}]
table[row sep=crcr]{%
	27	0.313114070518793\\
	65	0.194104755659312\\
	177	0.116384350190149\\
	545	0.0704139099719452\\
	1857	0.0431187508641578\\
	6785	0.0266400152224665\\
	25857	0.0165628965949829\\
	100865	0.0103426301986902\\
	398337	0.00647772548539351\\
};
\addlegendentry{error (uniform)}

\addplot [color=gray, dashed,forget plot]
table[row sep=crcr]{%
	7e3 0.072296659445794\\
	1e5 0.026670428643266\\
};

\addplot [color=mycolor1, line width=1.3pt, mark=*, mark options={solid, mycolor1}, mark repeat=3]
  table[row sep=crcr]{%
27	1.53306021270245\\
33	1.43254797101003\\
35	1.25489503230261\\
42	1.18468418471591\\
49	1.11780872220009\\
58	1.10562261276571\\
63	1.04155670400106\\
69	0.932113879537806\\
73	0.851704549225825\\
84	0.815866071010815\\
96	0.762832806104759\\
124	0.716444963788164\\
136	0.665340098502121\\
150	0.579031657749735\\
161	0.526118932488323\\
180	0.48619529499551\\
204	0.447889358812959\\
223	0.419993395491337\\
267	0.392656425104958\\
301	0.364505490985813\\
360	0.337051513395728\\
405	0.297368367988174\\
451	0.270749576701967\\
516	0.252250933291299\\
584	0.235194170401638\\
657	0.217273875339081\\
768	0.201211960407704\\
888	0.187305737127636\\
1057	0.173549919525605\\
1201	0.155619099371066\\
1337	0.141197612596924\\
1525	0.130826981903002\\
1762	0.122340969345807\\
2028	0.113607797276242\\
2329	0.105180181665606\\
2706	0.0979861226880848\\
3117	0.0914570418347433\\
3687	0.0842828593791107\\
4178	0.0757576225744867\\
4741	0.0693157692007317\\
5507	0.0648501749563534\\
6359	0.0610195475284337\\
7303	0.0568603211606501\\
8375	0.0526519283857913\\
9703	0.049200960737072\\
11170	0.046002513122629\\
13111	0.042597381646028\\
14998	0.0384175991016081\\
17037	0.0351348298062461\\
19837	0.0328463094594632\\
23059	0.0309438879942556\\
26580	0.0289867998839556\\
30436	0.0269206720018173\\
35118	0.025069346901472\\
40301	0.0235074686251624\\
47067	0.0219014704028087\\
54351	0.0199588281123192\\
61954	0.0181088416755798\\
71412	0.0168494634466185\\
83413	0.0158449381095434\\
96466	0.0149264946186832\\
110559	0.0139338701613326\\
126910	0.0129322315322831\\
146015	0.012139159042802\\
168567	0.0113613008720407\\
196531	0.0105316837442832\\
225221	0.0095012278855756\\
};
\addlegendentry{estimator (adaptive)}

\addplot [color=mycolor2, line width=1.3pt, mark=triangle*, mark options={solid, mycolor2}, mark repeat=3]
  table[row sep=crcr]{%
27	0.313114070518793\\
33	0.314723411242514\\
35	0.301571973987125\\
42	0.296454064403954\\
49	0.290407689263066\\
58	0.296236234255111\\
63	0.269731926956303\\
69	0.232898220702279\\
73	0.219853365545605\\
84	0.217815796144905\\
96	0.204751625977653\\
124	0.195154307246473\\
136	0.184640180271321\\
150	0.150614782594236\\
161	0.131414618178638\\
180	0.121060710635628\\
204	0.110013370393168\\
223	0.104928688595204\\
267	0.0993256235991521\\
301	0.0922922113238216\\
360	0.0881165401097593\\
405	0.0741909476857476\\
451	0.0657663655271358\\
516	0.061764276325863\\
584	0.0584819372095427\\
657	0.0536632370976868\\
768	0.050020807368561\\
888	0.0472145962876381\\
1057	0.0443880368185931\\
1201	0.0385749785496499\\
1337	0.0341412283759291\\
1525	0.0315498239767752\\
1762	0.0299854502263783\\
2028	0.0278244694278027\\
2329	0.0258075468824164\\
2706	0.0242956513050244\\
3117	0.0231226611678283\\
3687	0.021347929531997\\
4178	0.0185545929306835\\
4741	0.0167485041243796\\
5507	0.0157214957081293\\
6359	0.0149910834949771\\
7303	0.0139444631593981\\
8375	0.012875692289712\\
9703	0.0121983309555166\\
11170	0.011613041496454\\
13111	0.0107843869177473\\
14998	0.00944277308546364\\
17037	0.00849299423435781\\
19837	0.00794689996540153\\
23059	0.00757639709803888\\
26580	0.00711056332930039\\
30436	0.00658452920554429\\
35118	0.00616549114993633\\
40301	0.00590322385982341\\
47067	0.00554628066288303\\
54351	0.00496154831498915\\
61954	0.00438791094462142\\
71412	0.00406628960124434\\
83413	0.00384999513472366\\
96466	0.00366194791305118\\
110559	0.00340964893608477\\
126910	0.00315701127086957\\
146015	0.00302170755141426\\
168567	0.00286683847575419\\
196531	0.00265957280188931\\
225221	0.00232664777355307\\
};
\addlegendentry{error (adaptive)}

\addplot [color=gray, forget plot]
  table[row sep=crcr]{%
	7e3	0.023904572186688\\
	1e5 0.006324555320337\\
};

\end{axis}

\end{tikzpicture}%
	\caption{Convergence history for the stationary model problem of Section~\ref{sect:numerics:Lshape} using piecewise linear elements. The gray lines indicate orders~$0.375$ (dashed) and~$0.5$ (solid).}
	\label{fig:stationary_Lshape}
\end{figure}
%


Second, we consider a piecewise quadratic approximation for $u$ and $p$ in combination with a piecewise constant approximation of the Lagrange multiplier, i.e., 
\begin{equation}
\label{eqn:P2scheme}
	\Vh = \calP_2(\calT_{\Omega}), \qquad 
	\Qh = \calP_2(\calT_{\Gamma}), \qquad 
	\Mh = \calP^\text{d}_0(\calT_{\Omega}|_{\Gamma}). 
\end{equation}
Note that this is indeed a stable scheme according to the derivations of Section~\ref{sect:formulation:space}. The convergence results are shown in Figure~\ref{fig:stationary_Lshape_P2} and demonstrate an even higher numerical gain of the adaptive procedure. If the discrete multiplier space~$\Mh$ is set to~$\calP_1(\calT_{\Omega}|_\Gamma)$ or~$\calP_2(\calT_{\Omega}|_\Gamma)$, we observe the same rates. Hence, we omit the corresponding error plots here. 
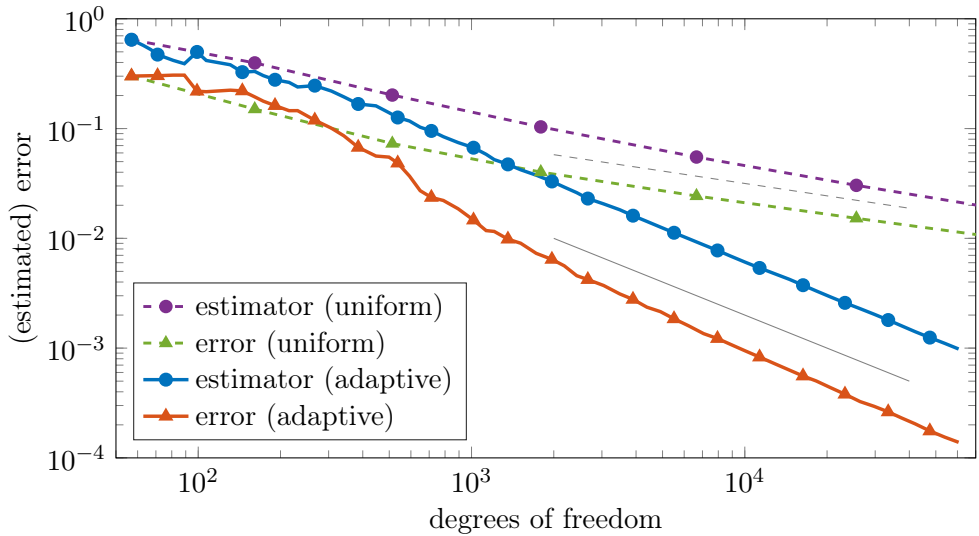
\begin{figure}
%
%
\definecolor{mycolor1}{rgb}{0.00000,0.44700,0.74100}%
\definecolor{mycolor2}{rgb}{0.85000,0.32500,0.09800}%
\definecolor{mycolor3}{rgb}{0.49400,0.18400,0.55600}%
\definecolor{mycolor4}{rgb}{0.46600,0.67400,0.18800}%
\begin{tikzpicture}

\begin{axis}[%
width=4.5in,
height=2.3in,
at={(2.5in,1.101in)},
scale only axis,
xmode=log,
xmin=50,
xmax=7e4,
xminorticks=true,
xlabel={degrees of freedom},
ymode=log,
ymin=1e-04,
ymax=1,
yminorticks=true,
ylabel={(estimated) error},
axis background/.style={fill=white},
legend style={legend cell align=left, align=left, at={(0.02,0.04)}, anchor=south west, draw=white!15!black}
]
\addplot [color=mycolor3, line width=1.1pt, dashed, mark=*, mark options={solid, mycolor3}]
table[row sep=crcr]{%
	57	0.644799373815\\
	161	0.396506675922182\\
	513	0.201943460743727\\
	1793	0.103551942377226\\
	6657	0.0550332447151835\\
	25601	0.0303901230705194\\
	100353	0.017356509839451\\
	397313	0.0101879305959064\\
};
\addlegendentry{estimator (uniform)}

\addplot [color=mycolor4, line width=1.1pt, dashed, mark=triangle*, mark options={solid, mycolor4}]
table[row sep=crcr]{%
	57	0.300817139448579\\
	161	0.150414547448613\\
	513	0.0730484589401183\\
	1793	0.040150737327292\\
	6657	0.0243124639790219\\
	25601	0.0152240842477965\\
	100353	0.0095989502000785\\
	397313	0.00605447729555946\\
};
\addlegendentry{error (uniform)}

\addplot [color=gray, dashed,forget plot]
table[row sep=crcr]{%
	2e3		0.057824748377162\\
	4e4		0.018803015465432\\
};

\addplot [color=mycolor1, line width=1.3pt, mark=*, mark options={solid, mycolor1}, mark repeat=3]
  table[row sep=crcr]{%
57	0.644799373815\\
63	0.571491648070258\\
67	0.524916559392508\\
71	0.47186969118717\\
81	0.417931887437756\\
89	0.389796622576729\\
99	0.498807618864767\\
107	0.417163959737229\\
131	0.38128847878461\\
145	0.326645141795959\\
161	0.331674113124509\\
173	0.303644913318356\\
191	0.27834192021163\\
215	0.263361381107737\\
231	0.239298244088975\\
267	0.245999656424635\\
307	0.219670540212529\\
341	0.194435301632709\\
385	0.167469828212291\\
447	0.161026023127041\\
499	0.139525335143329\\
537	0.126132306693875\\
593	0.117166755761357\\
649	0.102841368190135\\
713	0.0951736509295765\\
789	0.0843280326545937\\
893	0.0745933838775295\\
1017	0.0669369625542257\\
1131	0.0582845990117969\\
1211	0.0521500812154389\\
1359	0.0471744801097749\\
1513	0.0421440457135401\\
1725	0.0376211461793596\\
1969	0.032976698636113\\
2185	0.0292721530757265\\
2427	0.0260665540050735\\
2665	0.023034701061823\\
3045	0.0205882219400981\\
3489	0.0182067058774718\\
3899	0.0160871954153876\\
4385	0.0142511120314513\\
4911	0.0126636026851072\\
5509	0.0112435021902923\\
6229	0.00992511772357579\\
7067	0.0087235396197679\\
7945	0.00775616394041329\\
8897	0.00686903482933977\\
10021	0.00607136454391675\\
11325	0.00537349949999592\\
12877	0.00475524172860854\\
14555	0.00423328583235043\\
16333	0.00374293421813739\\
18245	0.00332071926255447\\
20605	0.00292353856177568\\
23257	0.00258410016666204\\
26405	0.00228823041654015\\
29747	0.0020343613292212\\
33481	0.00180221605925017\\
37485	0.00159544955424697\\
42013	0.0014116337105788\\
47545	0.00124682983115432\\
53869	0.00110560887445301\\
60597	0.000981446208626655\\
};
\addlegendentry{estimator (adaptive)}

\addplot [color=mycolor2, line width=1.3pt, mark=triangle*, mark options={solid, mycolor2}, mark repeat=3]
  table[row sep=crcr]{%
57	0.300817139448579\\
63	0.302091839288534\\
67	0.303057467107919\\
71	0.302958227522241\\
81	0.306708256234206\\
89	0.306634261120974\\
99	0.218901341472411\\
107	0.216666883305152\\
131	0.223678809493572\\
145	0.220380130007935\\
161	0.194413895243603\\
173	0.1776735303319\\
191	0.161929816343009\\
215	0.145524674378984\\
231	0.145644516398407\\
267	0.119253906762554\\
307	0.10085496878509\\
341	0.0859443333441144\\
385	0.0674215936979642\\
447	0.0561893145511645\\
499	0.0549502655065281\\
537	0.0484826440116618\\
593	0.0369183266960325\\
649	0.0271549104049491\\
713	0.0235836776375621\\
789	0.022120634001264\\
893	0.0186030617985743\\
1017	0.0146768068042312\\
1131	0.0117652966548274\\
1211	0.0115578200504207\\
1359	0.0098325623638185\\
1513	0.00902446870665117\\
1725	0.00730072382722446\\
1969	0.00643185294128793\\
2185	0.00561916848189454\\
2427	0.00456774281238024\\
2665	0.00421077944771773\\
3045	0.00371814481853705\\
3489	0.00311540002447035\\
3899	0.00277914560422406\\
4385	0.00235920537974272\\
4911	0.00214983277349669\\
5509	0.00184948458272923\\
6229	0.00159617461883935\\
7067	0.00135491036146912\\
7945	0.0012225099666419\\
8897	0.00107780978038903\\
10021	0.000945236496815259\\
11325	0.000829027870977384\\
12877	0.000720333287807368\\
14555	0.000631063821525644\\
16333	0.000557032646738909\\
18245	0.000502332196270759\\
20605	0.000436043592762939\\
23257	0.000380423022519804\\
26405	0.000325905159691058\\
29747	0.000294744920468895\\
33481	0.00026248444475878\\
37485	0.000230923193286995\\
42013	0.000204204461291913\\
47545	0.000175970731361624\\
53869	0.000154617968116448\\
60597	0.000138633837651818\\
};
\addlegendentry{error (adaptive)}

\addplot [color=gray, solid, forget plot]
  table[row sep=crcr]{%
2e3		1e-2\\
4e4		5e-4\\
};

\end{axis}

\end{tikzpicture}%
	\caption{Convergence history for the stationary model problem of Section~\ref{sect:numerics:Lshape} using piecewise quadratic elements for $u$, $p$ and piecewise constants for $\lambda$. The gray lines indicate orders~$0.375$ (dashed) and~$1.0$ (solid).}
	\label{fig:stationary_Lshape_P2}
\end{figure}

An illustration of the adaptive mesh refinement is given in Figure~\ref{fig:mesh_Lshape}. Note that both meshes show refinements in the same regions. Nevertheless, $u$ restricted to the boundary has~$190$ degrees of freedom, whereas $p$ is defined on a mesh with~$1030$ degrees of freedom. This additional refinement of the boundary is only possible because of the special formulation of the system equations as coupled system. Without the separation of $u|_\Gamma$ and $p$, we would need a refinement of $\calT_\Omega|_\Gamma$ in order to get the same accuracy. This, however, would call for an significant increase of the degrees of freedom due to the higher topological dimension of the bulk. 
\begin{figure}
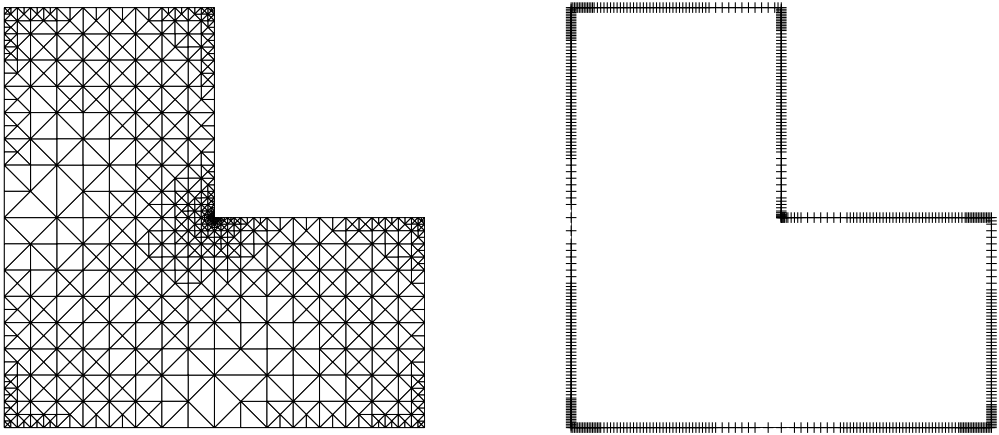

	\centering
	\include{pics/mesh_Lshape}
	\caption{Illustration of the bulk mesh $\calT_\Omega$ (left) and the boundary mesh $\calT_\Gamma$ (right) after 40 adaptive refinement steps.}
	\label{fig:mesh_Lshape}
\end{figure}
%
%
%
\subsection{Parabolic problem with dynamic boundary conditions}
\label{sect:numerics:dynamical}

Within this final numerical example, we consider the dynamic problem
\begin{align*}
	\dot u -\Delta u 
	&= 0.1 \hspace{3.65cm}\text{in } \Omega = (0,1)^2,\\
	\dot u - \Delta_{\Gamma} u + \partial_n u 
	&= xy \cos(\pi t x)\cos(\pi t y) \qquad\text{on } \Gamma = \partial\Omega.
\end{align*}
with homogeneous initial data. For the computation, we set $[1,10]$ as time horizon and apply the implicit Euler scheme with constant step size $\tau = 1.5\cdot 10^{-2}$ for the time stepping. For the spatial discretization, we consider piecewise quadratic elements as in~\eqref{eqn:P2scheme}. 
At every time point, we check if the estimated spatial error is smaller than $10^{-6}$. If not, we use the mentioned D\"orfler marking strategy with parameter~$\theta = 0.75$ and refine the meshes until this condition is satisfied. In the beginning, we start with~$41$, $16$, and $8$ degrees of freedom for $u$, $p$, and $\lambda$, respectively. 
At the end of the simulation, the numbers are~$4912$, $20674$, and $301$. The development of the degrees of freedom over time is illustrated in Figure~\ref{fig:dynamical_dofs} (left). 
\begin{figure}
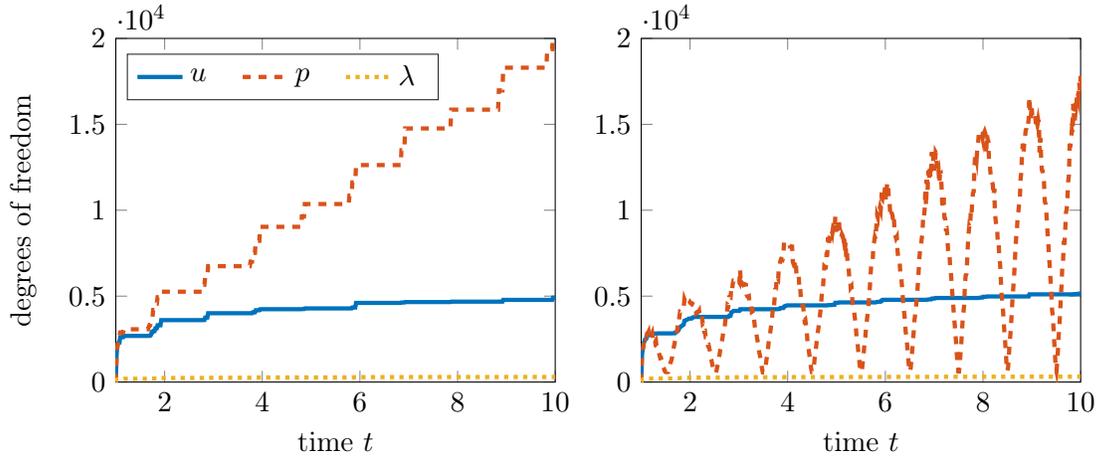

	\centering
	\include{pics/dynamical_dofs+adap}
	\caption{Evolution of the degrees of freedom for $u$, $p$, and $\lambda$ over time without (left) and with (right) coarsening of $\Qh$.}
	\label{fig:dynamical_dofs}
\end{figure}
%

As a proof of concept, we repeat this experiment with a coarsening strategy for the boundary mesh. Coarsening is of special interest for time-dependent problems, where the regions with large (estimated) errors can change rapidly~\cite[Sec.~4.4.4]{Bar16}. Here, we follow the heuristic strategy of~\cite[Sec.~7.2]{FunPW11}, which means that elements $I\in \calT_{\Gamma}$ are coarsened if
\[
	\widetilde{\eta}_I 
	\le \frac{1}{4\cdot 10^{6}\, |\calT_\Gamma|}.
\]
The associated degrees of freedom over time are shown in Figure~\ref{fig:dynamical_dofs} (right). Obviously, the numbers are always smaller compared to the previous simulation and are minimal for~$t \approx 1.5,2.5,\ldots$, i.e., whenever $g(t)=0$. This coincides with the times where the degrees of freedom of $p$ stay almost constant in the simulation without coarsening. At these points, a sound representation of $p$ on a coarse mesh is possible. 

Finally, we present the approximated solution~$u$ at the final time point $t=10$ in Figure~\ref{fig:dynamical_final}. One can see that the solution is oscillating strongly at the boundary parts~$(0,1)\times \{1\}$ and~$\{1\}\times(0,1)$. To capture this behavior numerically, we need more degrees of freedom for $p$. On the other hand, the approximation of $u$ gets along on a coarser mesh, since the solution is not so oscillatory in the bulk. 
\begin{figure}
	\centering
	\includegraphics[scale=0.8]{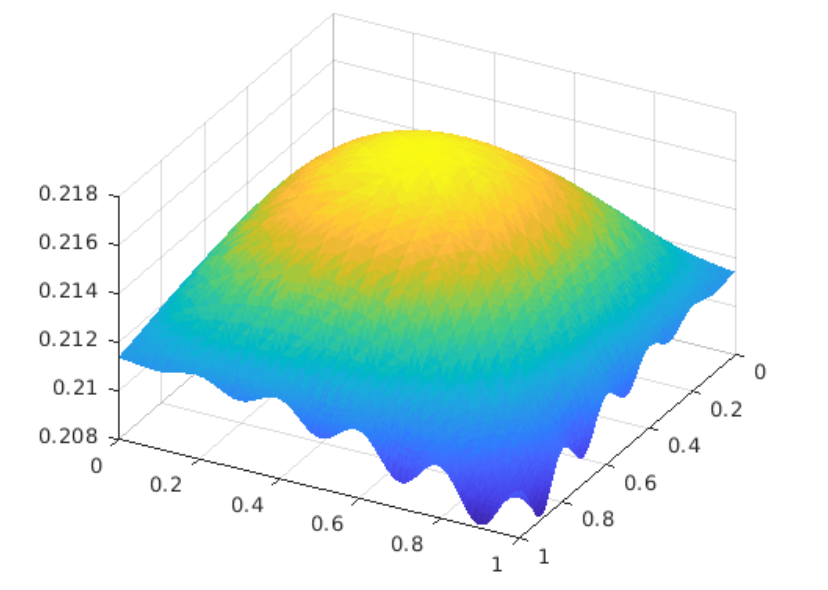}
	\caption{Numerical solution $u$ at the final time point $T=10$ for the parabolic problem of Section~\ref{sect:numerics:dynamical}}
	\label{fig:dynamical_final}
\end{figure}
%
%
\section{Conclusion}\label{sect:conclusion}
In this paper, we have introduced an adaptive procedure for the simulation of parabolic problems with dynamic boundary conditions. Based on a PDAE formulation, we are able to construct local error estimators, which distinguish necessary refinements of the bulk mesh along the boundary, i.e., of~$\calT_\Omega|_\Gamma$, and the boundary mesh~$\calT_\Gamma$. Moreover, we have characterized inf--sup stable finite element schemes for the spatial discretization. Within a number of numerical experiments, we have shown that the proposed algorithm reaches optimal convergence rates. In addition, one can observe that the boundary mesh gets more refined than $\calT_\Omega|_\Gamma$, which yields notable computational savings. 
%
%
\bibliographystyle{alpha} 
\bibliography{bib_dynBC}

\begin{thebibliography}{FGGR02}

\bibitem[AB17]{AcoB17}
G.~Acosta and J.~P. Borthagaray.
\newblock A fractional {L}aplace equation: {R}egularity of solutions and finite
  element approximations.
\newblock {\em SIAM J. Numer. Anal.}, 55(2):472--495, 2017.

\bibitem[AF03]{AdaF03}
R.~A. Adams and J.~J.~F. Fournier.
\newblock {\em Sobolev Spaces}.
\newblock Elsevier/Academic Press, Amsterdam, second edition edition, 2003.

\bibitem[AKZ23]{AltKZ22}
R.~Altmann, B.~Kov\'{a}cs, and C.~Zimmer.
\newblock Bulk--surface {L}ie splitting for parabolic problems with dynamic
  boundary conditions.
\newblock {\em IMA J. Numer. Anal.}, 2(43):950--975, 2023.

\bibitem[Alt15]{Alt15}
R.~Altmann.
\newblock {\em {R}egularization and {S}imulation of {C}onstrained {P}artial
  {D}ifferential {E}quations}.
\newblock Ph{D} thesis, Technische Universit{\"a}t Berlin, 2015.

\bibitem[Alt19]{Alt19}
R.~Altmann.
\newblock A {PDAE} formulation of parabolic problems with dynamic boundary
  conditions.
\newblock {\em Appl. Math. Lett.}, 90:202--208, 2019.

\bibitem[AO97]{AinT97}
M.~Ainsworth and J.~T. Oden.
\newblock A posteriori error estimation in finite element analysis.
\newblock {\em Comput. Methods Appl. Mech. Engrg.}, 142(1):1--88, 1997.

\bibitem[AV21]{AltV21}
R.~Altmann and B.~Verf\"u{}rth.
\newblock A multiscale method for heterogeneous bulk--surface coupling.
\newblock {\em Multiscale Model. Simul.}, 19(1):374--400, 2021.

\bibitem[AZ23a]{AltZ22ppt_c}
R.~Altmann and C.~Zimmer.
\newblock Dissipation-preserving discretization of the {C}ahn--{H}illiard
  equation with dynamic boundary conditions.
\newblock {\em Appl. Numer. Math.}, 190:254--269, 2023.

\bibitem[AZ23b]{AltZ22ppt_b}
R.~Altmann and C.~Zimmer.
\newblock Second-order bulk--surface splitting for parabolic problems with
  dynamic boundary conditions.
\newblock {\em IMA J. Numer. Anal}, 2023.
\newblock published online.

\bibitem[Bar16]{Bar16}
S.~Bartels.
\newblock {\em Numerical approximation of partial differential equations}.
\newblock Springer, Cham, 2016.

\bibitem[BF91]{BreF91}
F.~Brezzi and M.~Fortin.
\newblock {\em Mixed and Hybrid Finite Element Methods}.
\newblock Springer, New York, NY, 1991.

\bibitem[BPV00]{BraPV00}
J.~H. Bramble, J.~E. Pasciak, and P.~S. Vassilevski.
\newblock Computational scales of {S}obolev norms with application to
  preconditioning.
\newblock {\em Math. Comp.}, 69(230):463--480, 2000.

\bibitem[Bra07]{Bra07}
D.~Braess.
\newblock {\em Finite Elements - Theory, Fast Solvers, and Applications in
  Solid Mechanics}.
\newblock Cambridge University Press, New York, NY, third edition, 2007.

\bibitem[Car06]{Car06}
C.~Carstensen.
\newblock Cl{\'e}ment interpolation and its role in adaptive finite element
  error control.
\newblock In E.~Koelink, J.~van Neerven, B.~de~Pagter, G.~Sweers, A.~Luger, and
  H.~Woracek, editors, {\em Partial Differential Equations and Functional
  Analysis: The Philippe Cl{\'e}ment Festschrift}, pages 27--43, Basel, 2006.
  Birkh{\"a}user Verlag.

\bibitem[CFK23]{CsoFK23}
P.~Csomós, B.~Farkas, and B.~Kovács.
\newblock Error estimates for a splitting integrator for abstract semilinear
  boundary coupled systems.
\newblock {\em IMA J. Numer. Anal.}, published online, 2023.

\bibitem[Cl{\'{e}}75]{Cle75}
P.~Cl{\'{e}}ment.
\newblock Approximation by finite element functions using local regularization.
\newblock {\em RAIRO Anal. Numér.}, 2:77--84, 1975.

\bibitem[D\"96]{Dor96}
W.~D\"{o}rfler.
\newblock A convergent adaptive algorithm for {P}oisson’s equation.
\newblock {\em SIAM J. Numer. Anal.}, 33(3):1106--1124, 1996.

\bibitem[DS11]{DemS11}
A.~Demlow and R.~Stevenson.
\newblock Convergence and quasi-optimality of an adaptive finite element method
  for controlling $l_2$ errors.
\newblock {\em Numer. Math.}, 117:185--218, 2011.

\bibitem[EM13]{EmmM13}
E.~Emmrich and V.~Mehrmann.
\newblock Operator differential-algebraic equations arising in fluid dynamics.
\newblock {\em Comp. Methods Appl. Math.}, 13(4):443--470, 2013.

\bibitem[Esc93]{Esc93}
J.~Escher.
\newblock Quasilinear parabolic systems with dynamical boundary conditions.
\newblock {\em Commun. Part. Diff. Eq.}, 18(7-8):1309--1364, 1993.

\bibitem[Fai79]{Fai79}
G.~Fairweather.
\newblock On the approximate solution of a diffusion problem by {G}alerkin
  methods.
\newblock {\em J. Inst. Math. Appl.}, 24(2):121--137, 1979.

\bibitem[FGGR02]{FavGGR02}
A.~Favini, G.~R. Goldstein, J.~A. Goldstein, and S.~Romanelli.
\newblock The heat equation with generalized {W}entzell boundary condition.
\newblock {\em J. Evol. Equ.}, 2(1):1--19, 2002.

\bibitem[FPW11]{FunPW11}
S.~Funken, D.~Praetorius, and P.~Wissgott.
\newblock Efficient implementation of adaptive {P}1-{FEM} in {M}atlab.
\newblock {\em Comput. Methods Appl. Math.}, 11(4):460--490, 2011.

\bibitem[GGZ74]{GajGZ74}
H.~Gajewski, K.~Gr{\"o}ger, and K.~Zacharias.
\newblock {\em {N}ichtlineare {O}peratorgleichungen und
  {O}peratordifferential-{G}leichungen}.
\newblock Akademie-Verlag, Berlin, 1974.

\bibitem[Gol06]{Gol06}
G.~R. Goldstein.
\newblock Derivation and physical interpretation of general boundary
  conditions.
\newblock {\em Adv. Differential Equ.}, 11(4):457--480, 2006.

\bibitem[GT01]{GilT01}
D.~Gilbarg and N.~S. Trudinger.
\newblock {\em Elliptic Partial Differential Equations of Second Order}.
\newblock Springer-Verlag, Berlin, 2001.

\bibitem[HLR89]{HaiLR89}
E.~Hairer, C.~Lubich, and M.~Roche.
\newblock {\em The Numerical Solution of Differential-Algebraic Systems by
  {R}unge--{K}utta Methods}.
\newblock Springer-Verlag, Berlin, 1989.

\bibitem[HM12]{HipM12}
R.~Hiptmair and S.~Mao.
\newblock Stable multilevel splittings of boundary edge element spaces.
\newblock {\em BIT Numer. Math.}, 52:661--685, 2012.

\bibitem[HW96]{HaiW96}
E.~Hairer and G.~Wanner.
\newblock {\em Solving Ordinary Differential Equations {II}: Stiff and
  Differential-Algebraic Problems}.
\newblock Springer-Verlag, Berlin, second edition, 1996.

\bibitem[KL17]{KovL17}
B.~Kov{\'a}cs and C.~Lubich.
\newblock Numerical analysis of parabolic problems with dynamic boundary
  conditions.
\newblock {\em IMA J. Numer. Anal.}, 37(1):1--39, 2017.

\bibitem[Las02]{Las02}
I.~Lasiecka.
\newblock {\em Mathematical Control Theory of Coupled {PDE}s}.
\newblock Society for Industrial and Applied Mathematics (SIAM), Philadelphia,
  PA, 2002.

\bibitem[Lip04]{Lip04}
M.~K. Lipinski.
\newblock {\em {A} posteriori {F}ehlersch{\"a}tzer f{\"u}r
  {S}attelpunktsformulierungen nicht-homogener {R}andwertprobleme}.
\newblock Ph{D} thesis, Ruhr {U}niversit\"at {B}ochum, 2004.

\bibitem[LM72]{LioM72}
J.-L. Lions and E.~Magenes.
\newblock {\em {N}on-Homogeneous Boundary Value Problems and Applications {I}}.
\newblock Springer-Verlag, New York, NY, 1972.

\bibitem[LMT13]{LamMT13}
R.~Lamour, R.~M{\"a}rz, and C.~Tischendorf.
\newblock {\em Differential-Algebraic Equations: A Projector Based Analysis}.
\newblock Springer-Verlag, Berlin, Heidelberg, 2013.

\bibitem[LT85]{LanT85}
P.~Lancaster and M.~Tismenetsky.
\newblock {\em The Theory of Matrices: {W}ith Applications}.
\newblock Academic Press, Inc., San Diego, CA, second edition edition, 1985.

\bibitem[LW19]{LiuW19}
C.~Liu and H.~Wu.
\newblock An energetic variational approach for the {C}ahn–{H}illiard
  equation with dynamic boundary condition: {M}odel derivation and mathematical
  analysis.
\newblock {\em Arch. Rational Mech. Anal.}, 233:167--247, 2019.

\bibitem[SZ90]{ScoZ90}
L.~R. Scott and S.~Zhang.
\newblock Finite element interpolation of nonsmooth functions satisfying
  boundary conditions.
\newblock {\em Math. Comp.}, 54(190):483--493, 1990.

\bibitem[Tay11]{Tay11}
M.~E. Taylor.
\newblock {\em Partial Differential Equations I: Basic Theory}.
\newblock Springer, New York, NY, second edition edition, 2011.

\bibitem[Ver13]{Ver13}
R.~Verfürth.
\newblock {\em A Posteriori Error Estimation Techniques for Finite Element
  Methods}.
\newblock Oxford University Press, Oxford, 2013.

\bibitem[Vit18]{Vit18}
E.~Vitillaro.
\newblock On the the wave equation with hyperbolic dynamical boundary
  conditions, interior and boundary damping and source.
\newblock {\em J. Differ. Equations}, 265(10):4873--4941, 2018.

\bibitem[VS13]{VraS13a}
V.~Vr\'{a}bel' and M.~Slodi\v{c}ka.
\newblock Nonlinear parabolic equation with a dynamical boundary condition of
  diffusive type.
\newblock {\em Appl. Math. Comput.}, 222:372--380, 2013.

\bibitem[VV08]{VazV08}
J.~L. V{\'a}zquez and E.~Vitillaro.
\newblock Heat equation with dynamical boundary conditions of reactive type.
\newblock {\em Commun. Part. Diff. Eq.}, 33(4):561--612, 2008.

\bibitem[Zim21]{Zim21}
C.~Zimmer.
\newblock {\em Temporal {D}iscretization of {C}onstrained {P}artial
  {D}ifferential {E}quations}.
\newblock PhD thesis, Technische Universit{\"a}t Berlin, 2021.

\end{thebibliography}
\end{document}